\documentclass[a4paper,onecolumn, superscriptaddress,10pt,accepted=2023-06-02,issue=6, volume=5, shorttitle=papers]{compositionalityarticle}
 \pdfoutput=1
\usepackage[utf8]{inputenc}
\usepackage[english]{babel}
\usepackage[T1]{fontenc}
\usepackage{amsmath}
\usepackage{hyperref}
\usepackage{amsmath,amsthm,amssymb}
\usepackage[T1]{fontenc}
\usepackage{hyperref}
\usepackage{xurl}
\usepackage{tikz}
\usepackage{lipsum}

\usepackage{tikz}
\usepackage{cite}
\usepackage{circuitikz}
\usepackage[OT2, T1]{fontenc}
\usepackage{comment}
\usepackage{rotating}
\usepackage{graphicx}
\usepackage{pdflscape}
\usepackage{url}
\usepackage[symbol]{footmisc}

\newcommand\normal{{\cal N}}

\newcommand\matA{{\bf A}}

\newcommand\matH{{\bf H}}
\newcommand\matI{{\bf I}}
\newcommand\matQ{{\bf \Sigma}}
\newcommand\Diag{{\bf Diag}}
\newcommand\vecD{{\bf D}}

\newcommand\bfzero{{\bf 0}}
\newcommand\vecX{{\bf X}}
\newcommand\vecY{{\bf Y}}
\newcommand\vecv{{\bf v}}

\newcommand\vecy{{\bf y}}

\newcommand\Polya{P\'olya}

\newcommand\convD{{\buildrel {\cal D} \over \longrightarrow}}

\newcommand\Var{{\mathbb V{\rm ar}}}
\newcommand\Cov{{\mathbb C{\rm ov}}}
\newtheorem{theorem}{Theorem}[section]
\newtheorem{corollary}{Corollary}[section]

\newtheorem{example}{Example}[section]

\newtheorem{remark}{Remark}[section]
\newcommand\almostsure{\buildrel a.s. \over \longrightarrow}

\begin{document}
\title{Degrees in random $m$-ary hooking networks}
\date{}

%Ravi: Our ORCIDs have been updated.  It had Compositionality template placeholder values earlier. 
\author{Kiran R. Bhutani}
\email{bhutani@cua.edu}
\orcid{0000-0003-4823-0887}
\affiliation{Department of Mathematics, The Catholic University of America, Washington, D.C. 20064, U.S.A.}
\author{Ravi Kalpathy}
\email{kalpathy@cua.edu}
\orcid{0000-0001-5874-2427}
\affiliation{Department of Mathematics, The Catholic University of America, Washington, D.C. 20064, U.S.A.}
\author{Hosam Mahmoud}
\email{hosam@gwu.edu}
\orcid{0000-0003-0962-9406}
\affiliation{Department of Statistics, The George Washington University, Washington, D.C. 20052, U.S.A.}
\maketitle
\begin{abstract}
The theme in this paper is a composition of random graphs and \Polya\ urns. The random graphs are generated through a small structure called the seed.
Via \Polya\ urns,
we study the asymptotic 
degree structure in a random $m$-ary 
hooking network and identify strong laws.
We further upgrade the result to second-order asymptotics in the form of 
multivariate Gaussian 
limit laws. 
We give a few concrete examples 
and explore
some properties
with a full representation of the Gaussian limit in each case.
The asymptotic covariance matrix associated with the \Polya\ urn 
is obtained by a new method that originated in this paper and is reported in~\cite{Mahmoud2022}.
\end{abstract}
\section{Introduction}
Many random structures grow by adding small components. For example,
networks grow by linking new small structures, and urns grow by
adding balls, etc. In this manuscript, we deal with a certain flavor of network growth via ``hooking'' components of fixed structure (all additions are graphs that have the same
form). The basic repetitive form is called a {\em seed}, which has 
a particular node in it called the {\em hook}. 

Initially, the network is just the seed.
At every discrete point in time,
a copy of the seed is brought to the current network and adjoined to it by fusing
its hook into a node (vertex) chosen from the network. 
After a long period of growth, we have a large network. Figure~\ref{Fig:triangle}
shows the step-by-step growth of a hooking network grown from a triangular seed in four steps.
We have not specified how a node is chosen in the network. It could be specified deterministically by an external agent, such as an AI optimizer, or
it could be done according to a probability distribution on the nodes of
the existing network. In the latter case, the growth of the network is a stochastic process. In this light, Figure~\ref{Fig:triangle} would be
only one realization of the process.

\begin{figure}[thb]
%\begin{figure}[!htbp]
\begin{center}
\begin{tikzpicture}[scale=0.4]
\node[draw=white] at (1, 2.65) {hook};
%time n=0
\coordinate (A) at (0,0);
\coordinate (B) at (2,0);
\coordinate (C) at (1,1.732);

\draw [thick] (A)--(B)--(C)--cycle;

\draw [thick, fill=black] (0,0) circle [radius=0.2];
\draw [thick, fill=black] (2,0) circle [radius=0.2];
\draw [thick, fill=black] (1,1.732) circle [radius=0.2];
%========================================

%time n=1
\coordinate (A) at (4,0);
\coordinate (B) at (6,0);
\coordinate (C) at (5,1.732);
\coordinate (D) at (8,0);
\coordinate (E) at (7,1.732);

\draw [thick] (A)--(B)--(D)--(E)--(B)--(C)--cycle;

\draw [thick, fill=black] (6,0) circle [radius=0.2];
\draw [thick, fill=black] (8,0) circle [radius=0.2];
\draw [thick, fill=black] (7,1.732) circle [radius=0.2];

\draw [thick, fill=black] (4,0) circle [radius=0.2];
\draw [thick, fill=black] (6,0) circle [radius=0.2];
\draw [thick, fill=black] (5,1.732) circle [radius=0.2];

%===============================
%time n=2
\coordinate (A) at (10,0);
\coordinate (B) at (12,0);
\coordinate (C) at (11,1.732);
\coordinate (D) at (10,3.464);
\coordinate (E) at (12,3.464);

\draw [thick] (A)--(B)--(C)--(D)--(E)--(C)--cycle;

\draw [thick, fill=black] (10,0) circle [radius=0.2];
\draw [thick, fill=black] (12,0) circle [radius=0.2];
\draw [thick, fill=black] (11,1.732) circle [radius=0.2];
\draw [thick, fill=black] (10,3.464) circle [radius=0.2];
\draw [thick, fill=black] (12,3.464) circle [radius=0.2];

\coordinate (A) at (12,0);
\coordinate (B) at (14,0);
\coordinate (C) at (13,1.732);

\draw [thick] (A)--(B)--(C)--cycle;

\draw [thick, fill=black] (12,0) circle [radius=0.2];
\draw [thick, fill=black] (14,0) circle [radius=0.2];
\draw [thick, fill=black] (13,1.732) circle [radius=0.2];

%===============================

%time n=3
\coordinate (A) at (16,0);
\coordinate (B) at (18,0);
\coordinate (C) at (17,1.732);
\coordinate (D) at (16,3.464);
\coordinate (E) at (18,3.464);
\coordinate (F) at (20,0);
\coordinate (G) at (19,1.732);
\coordinate (H) at (21,1.732);
\coordinate (I) at (22,0);

\draw [thick] (A)--(B)--(C)--cycle;
\draw [thick] (C)--(E)--(D)--cycle;
\draw [thick] (F)--(B)--(G)--cycle;
\draw [thick] (F)--(H)--(I)--cycle;

\draw [thick, fill=black] (16,0) circle [radius=0.2];
\draw [thick, fill=black] (18,0) circle [radius=0.2];
\draw [thick, fill=black] (17,1.732) circle [radius=0.2];
\draw [thick, fill=black] (16,3.464) circle [radius=0.2];
\draw [thick, fill=black] (18,3.464) circle [radius=0.2];
\draw [thick, fill=black] (19,1.732) circle [radius=0.2];
\draw [thick, fill=black] (20,0) circle [radius=0.2];
\draw [thick, fill=black] (21,1.732) circle [radius=0.2];
\draw [thick, fill=black] (22,0) circle [radius=0.2];

%===============================

%time n=4
\coordinate (A) at (24,0);
\coordinate (B) at (26,0);
\coordinate (C) at (25,1.732);
\coordinate (D) at (24,3.464);
\coordinate (E) at (26,3.464);
\coordinate (F) at (28,0);
\coordinate (G) at (27,1.732);
\coordinate (H) at (29,1.732);
\coordinate (I) at (30,0);
\coordinate (J) at (25,-1.732);
\coordinate (K) at (27,-1.732);

\draw [thick] (A)--(B)--(C)--cycle;
\draw [thick] (C)--(E)--(D)--cycle;
\draw [thick] (F)--(B)--(G)--cycle;
\draw [thick] (F)--(H)--(I)--cycle;
\draw [thick] (J)--(K)--(B)--cycle;

\draw [thick, fill=black] (24,0) circle [radius=0.2];
\draw [thick, fill=black] (26,0) circle [radius=0.2];
\draw [thick, fill=black] (25,1.732) circle [radius=0.2];
\draw [thick, fill=black] (24,3.464) circle [radius=0.2];
\draw [thick, fill=black] (26,3.464) circle [radius=0.2];
\draw [thick, fill=black] (28,0) circle [radius=0.2];
\draw [thick, fill=black] (27,1.732) circle [radius=0.2];
\draw [thick, fill=black] (29,1.732) circle [radius=0.2];
\draw [thick, fill=black] (30,0) circle [radius=0.2];
\draw [thick, fill=black] (25,-1.732) circle [radius=0.2];
\draw [thick, fill=black] (27,-1.732) circle [radius=0.2];

\end{tikzpicture}
\end{center}
  \caption{A hooking network grown in four time steps from a triangle (seed).}
  \label{Fig:triangle}
\end{figure}
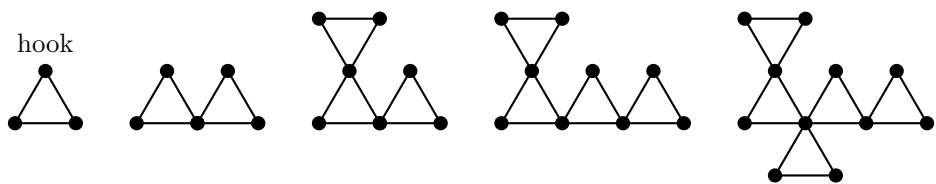

Alternative nomenclature appears in the jargon. For instance, 
adjoining the seed to the network is called by some hooking.
The word ``fusing'' also appears frequently. 
The chosen receptacle node in the graph to host the hook
of the newly added seed is referred to as a ``latch.'' The hooking (fusing)
operation is achieved by identifying the hook and the chosen latch. The act of adjoining
a seed at a step may then be considered as ``latching'' or ``recruiting''
by the latch. 
  
There has been recently a flurry of research papers on hooking networks (though some do not call it that)~\cite{Bahrani,Bhutani1,Bhutani2,Bhutani3,Bhutani4,
ChenChen,Holmgren,Colin1,Colin2,Gittenberger,Mohan,profile,SP,
Samorodnitsky}; see~\cite{VanderHofstad} for a comprehensive background.  

Applications can be found in organic chemistry, social networks and many other areas. For instance, in polymer chemistry, 
monomers link together into a covalently bonded network of macromolecules to form polymers~\cite{Bhutani1,Bhutani2,Bhutani4}. 
Social networks and the like grow by adding communities in which one member of a new community is
a friend of one member of an existing social network, where the degree of a node in the network is a reflection of the popularity of the 
network subscriber it represents.

With the degrees of nodes in a network being of prime importance, our 
purpose in this paper is to study the degrees in a certain kind of network (with
a limit on the number of hookings at every node).  As the degrees in the network and the balls in a certain kind of \Polya\ urns have the same generative mechanism,
we rely
on the urn composition to produce theorems for the network structure 
considered in this manuscript.

We note that the study of the
asymptotic degree structure in hooking networks has already been conducted  (under different assumptions) in  \cite{ChenChen,Holmgren,Colin1, profile}.
The authors of these papers did not impose a limit on the number
of hookings at a particular place.

Some applications put a cap on
the number of times a component can be added at one place. A popular example is
the binary search tree, which
arises in computer science as an efficient data structure that expedites
the fast retrieval of data~\cite{Brown,Knuth}. The binary imposition comes from
the physical implementation of the nodes in computer memory. A node is comprised 
of space for fields of information with distinguishable pointers to the left and right subtrees containing nodes of a similar structure. These pointers occupy different locations in the computer memory.

A few models of
hooking networks~\cite{Bhutani1,Holmgren,Colin1,Colin2} have been introduced for applications with no limit on the number
of hookings that can be made at a particular node or edge. In this paper,
we turn our attention to $m$-ary hooking networks, where at most $m$ (finite)
copies of a seed can join a node in the network. 
Figures~\ref{Fig:virtual}--\ref{Fig:extended}
will aid the reader visualize this construct.

Via \Polya\ urns, 
we study the asymptotic 
degree structure in these $m$-ary networks. 
We  identify strong laws
for the number of nodes of various degrees in the network. 
We further upgrade the result to second-order asymptotics in the form of 
multivariate Gaussian limit laws. 
We compute the covariances of
the \Polya\ urn by a new method that originated in this paper and is reported in~\cite{Mahmoud2022}. 
\subsection{The $m$-ary hooking network}
An $m$-ary network grows as follows: 
We start with a connected {\em seed} 
graph~$G_0$ 
with vertex set of size $\tau_0$.  
A vertex in the 
seed is designated as a {\em hook}.
Each copy of the seed that joins the network subsequently 
uses its own copy of that vertex for hooking. 
 At step $n$, a copy of the seed is hooked into the graph
 $G_{n-1} = (V_{n-1}, {\cal E}_{n-1})$ that exists at time $n-1$ to produce the graph $G_n = (V_n, E_n)$. So,
$\tau_0$ is the size of $V_0$, and
 $\tau_n = (\tau_0 -1) n +\tau_0$. 
For instantiation, 
see Figure~\ref{Fig:triangle} in which the seed is a triangle, and
$\tau_0 = 3$, $\tau_1= 5, \tau_2=7, \tau_3=9$, and $\tau_4 =11$.

We think of the graph $G_n$ as a network of {\em age} $n$.  
The hooking
is accomplished by fusing together the hook of a new copy of the seed and a
{\em latch} 
(vertex) randomly chosen from the network at age~$n-1$. A vertex in the network can qualify as a latch
at most $m\ge 1$ times. 

We consider a {\em uniform} probability model that equally likely selects any of the insertion
positions in the network as a latch.
Initially, there are $m$ insertion positions associated with each vertex in~$G_0$.
Every hooking takes away an insertion position from a vertex. 
After a node recruits $r \le m$ times, there remain $m-r$ insertion positions.
When all
$m$ positions are taken at a vertex (i.e.,  the vertex has recruited $m$ times),
the vertex is saturated and is no longer active in recruiting. 
 
A visual device helps us discern the insertion positions. We represent  
each insertion position at a node as a virtual node connected by an edge to the vertex. The graph carrying the virtual nodes is called the {\em extended network}. 
Initially, there are $m$ virtual nodes attached to each vertex of the seed,
which is the graph $G_0$. Figure~\ref{Fig:virtual} shows a seed to be used in building a binary network ($m=2$) and its extension. 
The nodes of the seed are shown as bullets and the virtual nodes in the extended
seed are shown as squares. 
The colors in the square nodes encode the degree  and history of the nodes
they are attached
to. Initially every node has two square nodes (insertion positions) attached to it and colors 
representing their degrees at the start. For instance, initially each node of degree~3
carries two white  (color 1) square nodes, and each node of degree~7 carries
two blue (color 2) square nodes. These colors will change over time according to a scheme
that is explained later.
    
\begin{figure}[thb]
%\begin{figure}[!htbp]
\begin{center}
\begin{tikzpicture}[scale=0.5]
\node[draw=white] at (-12, -2.55) {hook};
\node[draw=white] at (-12.55, 2.55) {$x$};
\node[draw=white] at (-7.45, 2.55) {$y$};
\node[draw=white] at (-7.45, -2.55) {$z$};
\coordinate (A) at (-12,-2);
\coordinate (B) at (-12,2);
\coordinate (C) at (-8,2);
\coordinate (D) at (-8,-2);
\draw [thick] (A)--(B)--(C)--(D)--cycle;
\draw [thick] (B)--(D);
\draw [thick] (A)--(C);

\draw [thick, fill=black] (-12,-2) circle [radius=0.2];
\draw [thick, fill=black] (-12,2) circle [radius=0.2];
\draw [thick, fill=black] (-8,2) circle [radius=0.2];
\draw [thick, fill=black] (-8,-2) circle [radius=0.2];
\draw [thick] (-7,2) ellipse (1 and 0.3);
\draw [thick] (-8,3) ellipse (0.3 and 1);

%%%%%%%%%%%%%%%%%%%%%%%%%%
\node[draw=white] at (-1, -2.55) {hook};
\coordinate (A) at (-2,-2);
\coordinate (B) at (-2,2);
\coordinate (C) at (2,2);
\coordinate (D) at (2,-2);
%%\coordinate (E) at (5,2.5);
%\coordinate (F) at (2.5,5);
\draw [thick] (A)--(B)--(C)--(D)--cycle;
\draw [thick] (A)--(C);
\draw [thick] (B)--(D);

\draw [thick, fill=black] (-2,-2) circle [radius=0.2];
\draw [thick, fill=black] (-2,2) circle [radius=0.2];
\draw [thick, fill=black] (2,2) circle [radius=0.2];
\draw [thick, fill=black] (2,-2) circle [radius=0.2];
\draw [thick] (3,2) ellipse (1 and 0.3);
\draw [thick] (2,3) ellipse (0.3 and 1);

\coordinate (E) at (-4,2);
\coordinate (F) at (4,-2);
\coordinate (G) at (-4,-2);
\coordinate (H) at (3.5,0.1);
\coordinate (I) at (-2,4);
\coordinate (J) at (-2,-4);
\coordinate (K) at (2,-4);
\coordinate (L) at (0.5,3.9);
\draw (B)--(E);
\draw (D)--(F);
\draw (A)--(G);
\draw (C)--(H);
\draw (A)--(J);
\draw (B)--(I);
\draw (C)--(L);
\draw (D)--(K);
\node at (-4,2) [rectangle,draw, fill=white] {$1$};
\node at (-2,-4) [rectangle,draw, fill=white] {$1$};
\node at (-2,4) [rectangle,draw, fill=white] {$1$};
\node at (4,-2) [rectangle,draw, fill=white] {$1$};
\node at (2,-4) [rectangle,draw, fill=white] {$1$};
\node at (-4,-2) [rectangle,draw, fill=white] {$1$};
\node at (3.85,0.15) [rectangle,draw, fill=cyan] {$2$};
\node at (0.5,4.3) [rectangle,draw, fill=cyan] {$2$};
\end{tikzpicture}
\end{center}
  \caption{A seed (left) for a binary network and its extension (right).}
  \label{Fig:virtual}
\end{figure}
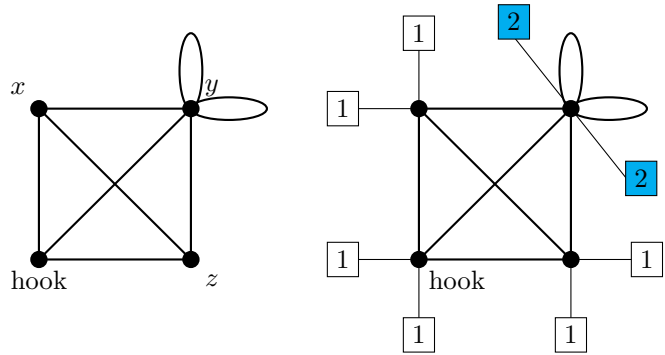

Figure~\ref{Fig:extended} shows a binary network grown 
from the seed in Figure~\ref{Fig:virtual}
in four steps. The four hookings occur at the nodes
labeled $x$ (twice), $y$ (once), and hook (once)  in Figure~\ref{Fig:virtual}. Note in Figure~\ref{Fig:extended} the appearance of additional colors (to be explained later). Note also that node $x$ no longer carries
virtual nodes, as it has used all its recruiting chances, having recruited twice.
There is no change at the node labeled $z$, but now each of the hook and $y$ 
carries only one virtual node, having only recruited once (with one extra 
chance left at recruiting). 

The seed in this particular binary hooking network
has nodes of degree~3 and 7. The hook is of degree~3. 
Hooking the seed into nodes of degree 3 increases the degree to 6, which can ultimately increase to 9;
hooking the seed into nodes
of degree 7 increases the degree to 10, which can ultimately increase to 13. 
The six degrees 3, 7, 6, 9, 10, 13 are the only admissible degrees. 
A coloring
scheme needs six tracking colors, one color for each admissible degree. 
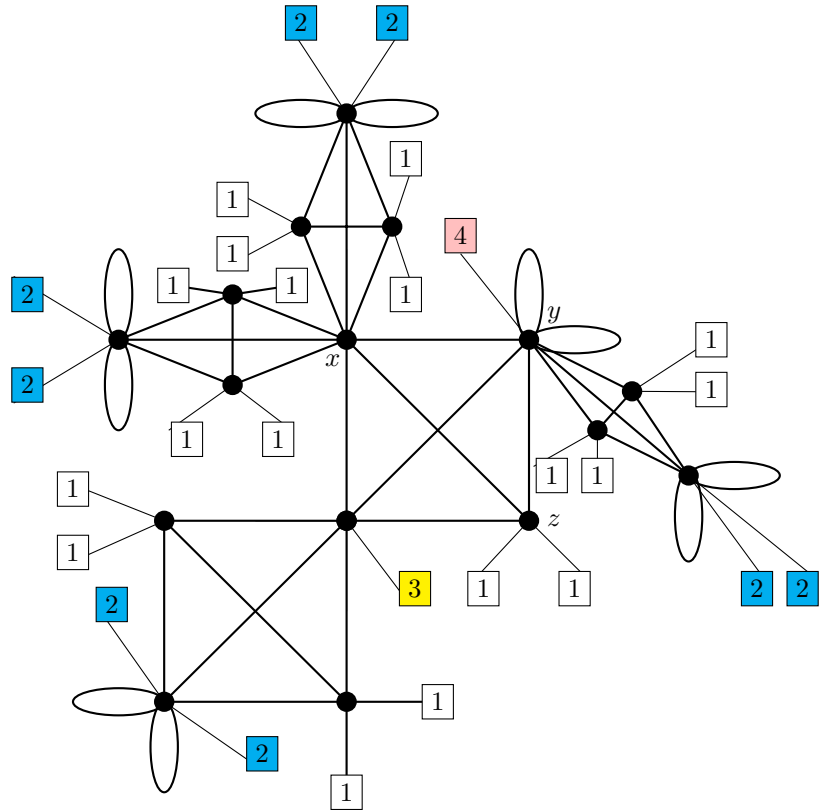
\begin{figure}[thb]
%\begin{figure}[!htbp]
\begin{center}
\begin{tikzpicture}[scale=0.6]

%%%%%%%%%%%%%%%%%%%%%%%%%%
\node[draw=white] at (-2.3, 1.55) {$x$};
\node[draw=white] at (2.55, 2.55) {$y$};
\node[draw=white] at (2.55, -2.0) {$z$};
\coordinate (A) at (-2,-2);
\coordinate (B) at (-2,2);
\coordinate (C) at (2,2);
\coordinate (D) at (2,-2);
\coordinate (E) at (3,-2);
\coordinate (F) at (1,-2);
\draw [thick] (A)--(B)--(C)--(D)--cycle;
\draw [thick] (A)--(C);
\draw [thick] (B)--(D);

\draw [thick, fill=black] (-2,-2) circle [radius=0.2];
\draw [thick, fill=black] (-2,2) circle [radius=0.2];
\draw [thick, fill=black] (2,2) circle [radius=0.2];
\draw [thick, fill=black] (2,-2) circle [radius=0.2];

\draw [thick, fill=black] (5.5,-1) circle [radius=0.2];
\draw [thick, fill=black] (4.26, 0.86) circle [radius=0.2];
\draw [thick, fill=black] (3.5,0) circle [radius=0.2];

\draw [thick] (3,2) ellipse (1 and 0.3);
\draw [thick] (2,3) ellipse (0.3 and 1);

\coordinate (E) at (-4,2);
\coordinate (F) at (3.3,-3.3);
\coordinate (G) at (-4,-2);
\coordinate (H) at (3.5,0.1);
\coordinate (I) at (-2,4);
\coordinate (J) at (-2,-4);
\coordinate (K) at (0.8,-3.3);
\coordinate (L) at (0.5,3.9);
\draw (D)--(F);
\draw (A)--(G);
\draw (A)--(J);
\draw (B)--(I);
\draw (C)--(L);
\draw (D)--(K);
\coordinate (R) at (3.5,0);
\coordinate (S) at (5.5,-1);
\coordinate (T) at (4.26, 0.86) ;
\draw [thick] (C)--(R)--(S)--(T)--cycle;
\draw [thick] (C)--(S);
\draw [thick] (R)--(T);

\coordinate (U) at (3.5, -1);
\coordinate (V) at (2.1, -0.8);
\coordinate (W) at (8.2,-3.2);
\coordinate (X) at (7.1,-3.2);
\coordinate (Y) at (6, 0.85);
\coordinate (Z) at (6, 2);
\draw (R)--(U);
\draw (R)--(V);
\draw (S)--(W);
\draw (S)--(X);
\draw (T)--(Y);
\draw (T)--(Z);

\node at (2.5,-1) [rectangle,draw, fill=white] {$1$};
\node at (3.5,-1) [rectangle,draw, fill=white] {$1$};
\node at (6,0.9) [rectangle,draw, fill=white] {$1$}; 
\node at (6,2) [rectangle,draw, fill=white] {$1$};
\node at (7,-3.5) [rectangle,draw, fill=cyan] {$2$};
\node at (8,-3.5) [rectangle,draw, fill=cyan] {$2$};
\node at (1,-3.5) [rectangle,draw, fill=white] {$1$};
\node at (3,-3.5) [rectangle,draw, fill=white] {$1$};
\node at (0.5,4.3) [rectangle,draw, fill=pink] {$4$};

%%%%%%%%%%%%%%%%%%%%%%%%%%%%%

\coordinate (A) at (-2,-2);
\coordinate (B) at (-6,-2);
\coordinate (C) at (-6,-6);
\coordinate (D) at (-2,-6);
\coordinate (E1) at (-8,-1.2);
\coordinate (E2) at (-8,-2.9);
\coordinate (F) at (-2,-8);
\coordinate (G) at (-2,-2);
\coordinate (H) at (0,-6);
\coordinate (I) at (-7.4,-4);
\coordinate (J) at (-4,-7.4);
\coordinate (K) at (-0.8,-3.6);
\draw [thick] (A)--(B)--(C)--(D)--cycle;
\draw [thick] (A)--(C);
\draw [thick] (B)--(D);
\draw (B)--(E1);
\draw (B)--(E2);
\draw [thick] (D)--(F);
\draw [thick] (D)--(H);
\draw  (C)--(I);
\draw  (C)--(J);
\draw  (A)--(K);
\draw [thick] (-7,-6) ellipse (1 and 0.3);
\draw [thick] (-6,-7) ellipse (0.3 and 1);

\draw [thick, fill=black] (-2,-2) circle [radius=0.2];
\draw [thick, fill=black] (-2,-6) circle [radius=0.2];
\draw [thick, fill=black] (-6,-2) circle [radius=0.2];
\draw [thick, fill=black] (-6,-6) circle [radius=0.2];

\node at (-0.5,-3.5) [rectangle,draw, fill=yellow] {$3$};
\node at (-8,-2.7) [rectangle,draw, fill=white] {$1$};
\node at (-8,-1.3) [rectangle,draw, fill=white] {$1$};
\node at (-2,-8) [rectangle,draw, fill=white] {$1$};
\node at (0,-6) [rectangle,draw, fill=white] {$1$};
\node at (-3.85,-7.15) [rectangle,draw, fill=cyan] {$2$};
\node at (-7.15,-3.85) [rectangle,draw, fill=cyan] {$2$};

%%%%%%%%%%%%%%%%%%%%%%%%%%%%%

\coordinate (A) at (-2,2);
\coordinate (B) at (-2,7);
\coordinate (C) at (-1,4.5);
\coordinate (D) at (-3,4.5);
\draw [thick, fill=black] (-2,7) circle [radius=0.2];
\draw [thick, fill=black] (-1,4.5) circle [radius=0.2];
\draw [thick, fill=black] (-3,4.5) circle [radius=0.2];
\coordinate (E) at (-7, 2);
\coordinate (F) at (-4.5, 3);
\coordinate (G) at (-4.5,1);
\coordinate (H) at (-4.3, 5.2);
\coordinate (I) at (-4.3,3.8);
\coordinate (J) at (-0.5,3);
\coordinate (K) at (-0.5,6);
\coordinate (L) at (-0.7,9);
\coordinate (M) at (-3.3,9);
\coordinate (N) at (-9.3,3.4);
\coordinate (O) at (-9.3,0.6);
\coordinate (P) at (-5.9, 0);
\coordinate (Q) at (-3.2,0);
\draw  (G)--(P);
\draw  (G)--(Q);
\draw  (D)--(H);
\draw  (D)--(I);
\draw  (C)--(J);
\draw  (C)--(K);
\draw  (B)--(L);
\draw  (B)--(M);
\draw  (E)--(N);
\draw  (E)--(O);

\draw [thick, fill=black] (-7, 2) circle [radius=0.2];
\draw [thick, fill=black] (-4.5, 1) circle [radius=0.2];
\draw [thick, fill=black] (-4.5, 3) circle [radius=0.2];
\draw [thick] (-3,7) ellipse (1 and 0.3);
\draw [thick] (-1,7) ellipse (1 and 0.3);
\draw [thick] (-7,3) ellipse (0.3 and 1);
\draw [thick] (-7,1) ellipse (0.3 and 1);

\draw [thick] (A)--(B)--(C)--(D)--cycle;
\draw [thick] (A)--(F)--(E)--(G)--cycle;
\draw [thick] (B)--(D);
\draw [thick] (A)--(C);
\draw [thick] (F)--(G);
\draw [thick] (A)--(E);
\node at (-9,3) [rectangle,draw, fill=cyan] {$2$};
\node at (-9,1) [rectangle,draw, fill=cyan] {$2$};

\node at (-3,9) [rectangle,draw, fill=cyan] {$2$};
\node at (-1,9) [rectangle,draw, fill=cyan] {$2$};

\node at (-0.7, 3) [rectangle,draw, fill=white] {$1$};
\node at (-0.7, 6) [rectangle,draw, fill=white] {$1$};

\node at (-4.5, 5.1) [rectangle,draw, fill=white] {$1$};
\node at (-4.5, 3.9) [rectangle,draw, fill=white] {$1$};

\node at (-5.5, -0.2) [rectangle,draw, fill=white] {$1$};
\node at (-3.5, -0.2) [rectangle,draw, fill=white] {$1$};

\draw [thick] (5.5,-1.9) ellipse (0.3 and 1);
\draw [thick] (6.5,-1) ellipse (1 and 0.3);

\coordinate (A) at (-4.5,3);
\coordinate (B) at (-3.2,3.2);
\coordinate (C) at (-5.8,3.2);
\draw [thick] (A)--(B);
\draw [thick] (A)--(C);
\node at (-3.2,3.2) [rectangle,draw, fill=white] {$1$};
\node at (-5.8,3.2) [rectangle,draw, fill=white] {$1$};
\end{tikzpicture}
\end{center}
  \caption{An extended binary network at age 4 (after four hookings). Colors 1, 2, 3, and 4 correspond to active nodes of degrees 3, 7, 6, and 10, respectively.}
  \label{Fig:extended}
\end{figure}
\subsection{Organization}
This section is continued in 
Subsection~\ref{Subsec:notation}, which sets up
the notation used throughout.
The sequel is organized in sections. 
We investigate the distribution of degrees in the network $G_n$, for large~$n$.
It is demonstrated that asymptotically the admissible degrees have joint
multivariate Gaussian laws.
A paramount device for this investigation is the \Polya\ urn~\cite{Mybook}. 
In Section~\ref{Sec:Polya}, we say a word about  \Polya\ urns and mention the
specific results on which our investigation falls back. 
In Section~\ref{Sec:urnunderlying}, we associate a specific \Polya\ urn
with $m$-ary networks.

Applying urn theory requires 
extensive linear
algebra computation for large matrices. We take up most of the computation to find
eigenvalues and eigenvectors in Section~\ref{Sec:urnunderlying}. 
The main results, presented in Section~\ref{Sec:Gaussian},
are strong laws for the degrees, 
and asymptotic Gaussian laws, for any $m\ge 1$.
In Section~\ref{Sec:examples}, we give some specific instances in
four subsections. The examples
illustrate the usual cases and 
properties that appear in
degenerate cases.  The lengthy details for the 
construction of the covariance matrix in the binary example in 
Section~\ref{Sec:examples}
are relegated to an appendix.
\subsection{Notation}
\label{Subsec:notation}
We use the notation $[r]$ to denote the set $\{1,2, \ldots, r\}$.
For a real number~$x$, 
we denote the falling factorial $x(x-1)\ldots (x-r+1)$ by $(x)_r$; the
usual interpretation of~$(x)_0$ is 1.

Suppose the distinct degrees in the seed are $d_1 < d_2 <\cdots <d_k$,
for some $k \ge 1$ , and there are $n_\ell$ vertices of degree 
$d_\ell$ in the seed, for $\ell = 1, \ldots, k$. Throughout, the hook is considered to be of degree
$d_i =:  h$, and we
assume there are $n_j\ge 1$ nodes of degree $d_j$ in the seed.
Hence, the degrees that appear in the $m$-ary network are $d_j+sh$, for $j\in[k]$
and for $s=0, 1, \ldots, m$. Note that the numbers $d_j+sh$ need not be 
distinct. We present an example of this 
in Subsection~\ref{Subsec:unarydegen}.

We print matrices and vectors in boldface, often subscripted in a way to
reveal the dimension or the position in a block matrix. The distinction
is clear from the context.
We use $\bfzero_j$ to denote either a vector of $j$ zeros or a $j\times j$ matrix of zeros. The intent will
be determined by the context.
We use $\matI_j$ for the $j\times j$ identity 
matrix.
For a matrix $\bf C$, we denote the transpose by~${\bf C}^T$; the notation $|\bf C|$ is its determinant. We refer to a diagonal matrix
that has the numbers $b_1, \ldots, b_s$ on the diagonal as $\Diag(b_1, \ldots, b_s)$. For a function~$g(n)$,
the notation ${\bf O}(g(n))$ and ${\bf o}(g(n))$ respectively 
represent a matrix in which all the entries
are $O(g(n))$ and $o(g(n))$ in the usual sense.
\subsection*{Acknowledgements}
The authors are thankful to Professor Alexander Boris Levin for a valuable
discussion on the linear algebra. We also thank two anonymous reviewers
for their insightful comments and helpful suggestions.
The referees' sincere and generous effort improved the exposition.
\section{\Polya\ urns}
\label{Sec:Polya}
A multicolor \Polya\ urn scheme is initially nonempty. 
Suppose up to $c$ colors,  numbered $1, 2, \ldots, c$,
can appear over the course of time.
At each time step, a ball is drawn at
random 
from
the urn and its color is noted. If the color of the ball
withdrawn is $i$, we put it back in the urn and add $a_{i,j}$ balls of 
color $j\in [c]$, and the drawing is continued. These 
dynamics are captured in a $c \times c$ replacement matrix:
$$\matA = \begin{pmatrix}
   a_{1,1} &a_{1,2} &\ldots&a_{1,c}\\
   a_{2,1} &a_{2,2} &\ldots&a_{2,c}\\
   \vdots &\vdots &\ddots&\vdots\\
   a_{c,1} &a_{c,2} &\ldots&a_{c,c}
\end{pmatrix}.$$

In a general setting, the entries of the replacement matrix
can be negative (which means removing balls) or even 
random. All the urns that appear in this paper have deterministic
entries. We restrict all further presentations to fixed replacement matrices. 
Let $X_{n,i}$ be the number of balls of color $i\in [c]$ in the urn after 
$n$ draws, and
let $\vecX_n$ be the random vector with these components. 
The urn scheme is said to be {\em balanced}, if the number
of balls added at each step is constant, say $\theta\ge 0$, (that is, $\sum_{j=1}^c a_{i,j} =\theta$, for all $i\in [c]$). The parameter $\theta$ is called the {\em balance factor}.
The scheme is called {\em tenable} if it is always possible to draw balls
and execute the rules. In a tenable scheme, the urn never becomes
empty and the scheme never calls for taking out balls of a color, 
when there is not a sufficient number of balls of that color present in the urn. 
We shall focus only
on balanced tenable urn schemes with deterministic replacement matrices,
as that is all we need in our investigation of $m$-ary hooking networks.

Suppose the $c$ eigenvalues of $\matA^T$ are $\lambda_1, \ldots,
\lambda_c$ and are labeled
according to decreasing order of their real parts. That is, we arrange the 
eigenvalues so that
$$\Re\, \lambda_1 \ge \Re\, \lambda_2 \ge \cdots \ge \Re\, \lambda_c.$$ 
The eigenvalue $\lambda_1$ is called the {\em principal
eigenvalue} and the corresponding eigenvector $\vecv_1$ is called
the {\em principal eigenvector}. In the context of urns, the scale
of $\vecv_1$ is chosen so that $||\vecv_1||_1$ is 1.
 For a tenable balanced 
urn scheme
with a deterministic replacement matrix, the principal eigenvalue
is the balance factor by Perron--Frobenius theorem, i.e., $\lambda_1 = \theta$;
see Chapter 8 in~\cite{Horn}.

The number of balls in the scheme follows a strong law~\cite{Athreya}:
\begin{equation}
 \frac 1 n\, \vecX_n \almostsure \lambda_1 \vecv_1.
\label{Eq:Athreya}
\end{equation}
Smythe~\cite{Smythe} and Janson~\cite{Janson} present a theory for classes of generalized urn models, wherein one finds joint Gaussian laws.

Under mild regularity conditions, 
strong laws and Gaussian limit laws for properly normalized
versions of $\vecX_n$ are available in terms of the eigenvalues
and eigenvectors of $\matA^T$. 
In practice, the urn schemes associated with the growth of random structures are often tenable, growing, balanced and the corresponding replacement matrices have finite entries. 
Further, most applications involve schemes in which no color is redundant, 
in the sense that from the starting condition every color appears infinitely
often, even if it is not present at the start.
For this restricted class of urn schemes,  the normality conditions are:
\begin{itemize}
\item [(1)] $\lambda_1 > 0$ is real and simple.
\item [(2)]  $\frac {\Re(\lambda_2)} {\lambda_1}$ 
is at most $\frac 1 2$.
\end{itemize}

These conditions are certainly met in the urn associated with an $m$-ary network.

The theory in~\cite{Smythe,Janson} covers urn schemes with random entries under
some rather general regularity conditions. 
These conditions are significantly simplified in
the case of tenable balanced deterministic replacement matrices. For the latter class, asymptotically
there is an underlying joint multivariate normal distribution, if~$\lambda_1$ is 
real positive and simple (of multiplicity 1) and   
all the components of $\vecv_1$ 
are positive with $\Re\, \lambda_2 < \frac 1 2 \, \Re \, \lambda_1$. In
such a case, we have
a multivariate Gaussian law:
\begin{equation}
\frac 1 {{\sqrt n}} (X_n -\lambda_1 \vecv_1) \, \convD \, \normal (\bfzero_c, {\bf \Sigma}_c), 
\label{Eq:Smythe}
\end{equation}
where the limit is a multivariate normal vector with mean vector $\bfzero_c$, and some $c\times c$  covariance matrix
${\bf \Sigma_c}$. 

The method in~\cite{Janson} for the computation of the covariance matrix associated with a \Polya\ urn is quite elaborate. 
More recently, several authors revisited these
issues from an existential point of view~\cite{newJanson,Pouyanne} and from
a computational point of view~\cite{Mahmoud2022}.

The general theory in~\cite{Janson} covers matrices with random entries.
According
to~\cite{newJanson}, in 
the case of small-index urns (with $\lambda_2 < \frac 1 2 \lambda_1$), 
$\frac 1 n\Cov[\vecX_n]$
converges to the limit matrix $c\times c$, and~\cite{Mahmoud2022} specifies
the limiting matrix~$\matQ_c$ as the one that solves a linear matricial equation.
Specialized to the case of deterministic \Polya\ urns, $\matQ_c$ is the solution to the equation
\begin{equation}
\lambda_1\matQ_c  = \matA^T \matQ_c + \matQ_c \matA + \lambda_1 \matA^T 
          \big(\Diag(x_1, \ldots, x_c)
             -  \vecv_1 \vecv_1^T\big) \matA ,
\label{Eq:Mahmoud21}
\end{equation}
where the diagonal matrix has the $c$ components of the principal
eigenvector~$\vecv_1 = (x_1, \ldots, x_c)^T$ on its diagonal.
\section{The  \Polya\ urn underlying $m$-ary hooking networks}
\label{Sec:urnunderlying}
We map the virtual nodes (corresponding to actively recruiting vertices) onto
a \Polya\ urn on $mk$ colors. This number of colors corresponds 
to the evolution of $k$ different initial degrees, and each node
of any particular initial degree can recruit $m$ times. 
Ultimately, this mapping leads us
to strong laws for active and inactive nodes in the network. 

Recall that $k$ is the number of distinct degrees in the seed,
$h$ is the degree of the hook, and $n_\ell$ is the number of 
vertices of degree $d_\ell$
in the seed, for $\ell = 1, \ldots, k$.
For $j\in [k]$ and $s =0, \ldots, m-1$, we associate
colors $j+sk$ respectively with {\em active} nodes that are originally of degree~$d_j$, and experience latching $s$ times.
 
As the node degrees increase by hooking, they lose insertion positions. The
fewer insertion positions left (if any) change colors. 

Think of
the virtual nodes as balls in a \Polya\ urn.
The urn evolves as follows: When a virtual node of color $j \not = i$ 
is chosen as the insertion position at time $n$, an extended seed is  
hooked to the latch it emanates from.  
(Recall that the hook is of degree $d_i$.)
We determine the replacement rules of the urn by
distinguishing several cases according to the color of the drawn ball.
Consider latching at a node of degree $d_j$, for $j\in [k]$.
For $\ell \in[k]$, $\ell \not = j$ and $\ell \not = i$,
the hooking
adds $mn_\ell$ virtual nodes of color $\ell$ to the extended network $G_n$ ($m n_\ell$ balls of color $\ell$ to the urn), resulting in adding $n_\ell$ actual nodes 
 of degree $d_\ell$ each to $G_n$.
The dynamics impose
two exceptions at $\ell = j \not = i $ and  at $\ell = i$. At $\ell = j$,  
the hooking adds only $mn_j -m$ virtual nodes of color $j$ (balls of color $j$ in the urn)
as $m$ virtual nodes of color $j$ from the graph $G_{n-1}$ are lost in $G_n$. Also, the incoming hook loses its $m$ virtual nodes, so we add only
$mn_i-m$ balls of color $i$. 

The rules for the case $j=i$ are similar. We add $mn_\ell$ nodes of color $\ell \not = i$.
The hooking occurs at a node of the hook degree $d_i$, with $m$ virtual nodes lost
from $G_{n-1}$ and $m$ more from the incoming seed. So, we only add $mn_i-2m$ virtual
balls of color $i$ to the urn.

In the binary example of Figures~\ref{Fig:virtual}--\ref{Fig:extended},
the seed has the degrees~3 and~7; the associated external nodes
are encoded with colors~1 and~2, respectively. Nodes of degree 3 can progress to be of degree~6, and nodes of degree 7 can progress to be of degree 10;
the corresponding external nodes are of colors 3 and 4, respectively.
Then again,  nodes of degree 6 can progress to be of degree 9, and nodes of degree 10 can progress to be of degree 13, meeting the quota of at most two hookings,
so, they carry no external nodes, and no colors are associated with them.
The scheme needs four colors in all.
\begin{remark}
\label{Rem:duplicate}
It is possible for some seed structures to have the numbers 
$d_j + rk$ and $d_{j'} + sk$
being equal. An example of this arises 
from a seed of
a binary network that has the two degrees $d_1=2$ and $d_2 = 4$, with $k=2$,
with a hook of degree $d_1=2$.
In this instance, $d_1 + 2 =4$ and $d_2=4$ (with $j=1, j' =2, r=1$ and $s=0$). 
There are two types of nodes of degree~4, depending on their recruiting history. 
Nodes of degree 2 grow to become of degree 4 and 6, and nodes of
degree 4 grow to become of degree 6 and 8. The urn gives a fine
ramification of nodes of a degree like 4, distinguishing their virtual
nodes (by colors)
as virtual nodes of color 2 (initially attached to nodes of degree 4), and
virtual nodes of color 3 (attached to nodes initially of color 1 and recruited 
once). The total number of nodes of degree 4, regardless of their history,
is a combination of both counts.  
\end{remark}
The replacement matrix is an $mk \times mk$ matrix represented in $m^2$ blocks.
That is, we have 
$\matA = [\matH_{i,j}]_{1\le i,j \le m}$, 
and each block $\matH_{i,j}$ is of size $k\times k$.
The foregoing discussion specifies the top left block:  
\begin{align*}
\matH_k  &=\matH_{1,1}\\
           &= m \begin{pmatrix}  n_1- 1&  n_2
                   & \ldots & n_{i-1}& n_i -1&
             n_{i+1}& \ldots& n_k\\
             n_1&  n_2-1& \ldots & n_{i-1}& n_i - 1&
             n_{i+1}& \ldots&  n_k\\
            \vdots &\vdots&\ddots&\vdots&\vdots&\vdots&\ddots&\\
             n_1& n_2& \ldots & n_{i-1} - 1&  n_i - 1&
             n_{i+1}& \ldots&  n_k\\
             n_1&  n_2& \ldots & n_{i-1}&  n_i - 2 &
             n_{i+1}& \ldots&  n_k\\
             n_1&  n_2& \ldots &n_{i-1}&  n_i -1&
             n_{i+1} - 1& \ldots&  n_k\\
             \vdots &\vdots&\ddots&\vdots&\vdots&\vdots&\ddots&\\
              n_1&  n_2& \ldots &n_{i-1}&  n_i -1&
              n_{i+1}& \ldots&  n_k-1
\end{pmatrix}.
\end{align*}

When a virtual node of color $j\in [k]$ is taken for
the next hooking, we replace~$m$ virtual nodes of
that color with $m-1$ virtual nodes of color $j+k$, and make no other changes.
Hence, we have the blocks
$$\matH_{1,2} = (m-1)\matI_k, \qquad\qquad \matH_{1,r} = \bfzero_k,\qquad \mbox {for\ } r = 3, \ldots, m.$$

If a virtual node of  color $j = r +  sk > k$, for $r\in[k]$, $s\in[m-1]$  
is the insertion position,
we add $m n_j$ virtual nodes of color $j$, for $j \not = i$ (they
come with the hooked seed), and add only $mn_i-m$ virtual nodes of color $i$,
as the extended seed loses its $m$ virtual node upon hooking. 
So, the blocks $\matH_{2,1},\matH_{3,1}, \ldots, \matH_{m,1}$ are all 
the same, and all equal to a matrix which we call~$\matH_k'$. The matrix
$\matH_k'$ is identical to $\matH_k$, except that its diagonal entries have an extra $m$. That is, we have $\matH'_k = \matH_k + m \matI_k$. 

If a ball of color $j$, for $j=1, \ldots, k$ is drawn,
we adjust the ball count of such a color according to the block
$\matH_{1,1}$. In addition, we increase the number of balls 
of color $j+k$ by $m-1$ as the remaining virtual nodes
that are siblings of the one taken now have a history of recruiting once.
In other words, the block $\matH_{1,k}$ is set to $(m-1)\matI_k$.

After recruiting $s$ times, for $1\le s \le m-1$, a node of original degree $d_j$ has $m-s$ external nodes attached to it.  These balls receive the color $j+sk$. 
When a ball of this color is drawn from the urn, 
the node it belongs to has now recruited
$s+1$ times. We have one more insertion position taken, leaving
only $m-s-1$ insertion positions (virtual nodes) attached 
to the owning node.
The remaining insertion positions change colors to reflect the fact that the node  
they belong to has recruited $(s+1)$ times. The action in the urn is 
to put the drawn ball back, 
take out $m-s$ balls of color $j+sk$
and add to the urn $m-s-1$ balls of color $j+ (s+1)k$.
We have explained how the blocks $\matH_{i,j}$, for $i=2,\ldots m-1$,
$j= 2, \ldots m$ are formed.

When a node recruits its $m$th seed, it becomes inactive. 
So, a node recruiting its last seed, gets lost. This is reflected as a $-1$
in the replacement matrix.
Thus, the bottom-right $k\times k$ block is the negative of a $k\times k$ identity matrix, any other block that is not named explicitly in the preceding explanation is set to ${\bf 0}_k$.

The full replacement matrix is

  \vskip 0.5cm 
$$\matA =\begin{tabular}{|c|c|c|c|c|c|c|}
  \hline
$\matH_k$ & $(m-1)\matI_k$  & $\bfzero_k$& $\bfzero_k$ & $\cdots$ & $\bfzero_k$ 
             &$\bfzero_k$ \\
   \hline
$\matH_k'$ & $-(m-1)\matI_k$  & $(m-2)\matI_k$ & $\bfzero_k$ &$\cdots$ & $\bfzero_k$ 
             &$\bfzero_k$\\
    \hline
$\matH_k'$ & $\bfzero_k$  & $-(m-2)\matI_k$& $(m-3)\matI_k$ &$\cdots$ & $\bfzero_k$ 
             &$\bfzero_k$\\
  \hline
$\matH_k'$ & $\bfzero_k$  & $\bfzero_k$& $-(m-3)\matI_k$&$\cdots$&$\bfzero_k$&$\bfzero_k$\\
\hline
&&&&&&\\
$\vdots$ & $\vdots$  & $\vdots$& $\vdots$& $\ddots$& $\vdots$& $\vdots$\\
&&&&&&\\
\hline
$\matH_k'$ & $\bfzero_k$  & $\bfzero_k$&$\bfzero_k$ &$\cdots$&$2\matI_k$&$\bfzero_k$\\
\hline
$\matH_k'$ & $\bfzero_k$  & $\bfzero_k$&$\bfzero_k$ &$\cdots$&$-2\matI_k$&$\matI_k$\\
\hline
$\matH_k'$ & $\bfzero_k$  & $\bfzero_k$&$\bfzero_k$ &$\cdots$&$\bfzero_k$&-$\matI_k$\\
\hline
\end{tabular} \ .
$$
Note that an urn scheme with such
a replacement matrix is
balanced. Hence,~(\ref{Eq:Mahmoud21}) applies for the calculation of covariances.
  \vskip 1cm 
\begin{example}
Consider a ternary network ($m=3$) built from the seed in Figure~\ref{Fig:virtual}. The associated replacement matrix is
\begin{align*}
 \begin{pmatrix}  
             3&  3 &2 &0& 0&0\\
             6&  0 &0 &2& 0&0\\
             6&  3 &-2 &0& 1&0\\
             6&  3 &0 &-2& 0&1\\
             6&  3 &0 &0& -1&0\\
             6&  3 &0 &0& 0&-1
\end{pmatrix}.
\end{align*}
\end{example}
\subsection{Eigenvalues}
To find the eigenvalues of $\matA^T$ (which are the same as those of $\matA$), we solve for the roots of the
characteristic polynomial $|\matA -\lambda\matI_{mk}|=0$.
Some block operations help us get through. 
We determine a rectangular block in $\matA$ by specifying the position of its top left cell and bottom right cell---the block $(i,j)$--$(k,\ell)$
 is comprised of all the cells $(p,q)$, with $i\le  p\le k$ and $j \le q \le \ell$. For example, $\matA$ itself is the block $(1,1)$--$(mk,mk)$.
Multiply each entry of 
the bottom $m\times m k$ block by $-1$ and add this block to all the blocks above it. More
precisely, 
multiply each entry of 
the block $((m-1)k+1,1)$--$(mk,mk)$ by $-1$, then
add the $i$th row of 
that block to row $i+rk$, for $r =0, \ldots, (m-1)$. 
%} 
This produces
the matrix
 
%{\tiny
{\scriptsize
\vskip 1cm
\begin{tabular}{|c|c|c|c|c|c|c|}
  \hline
$\matH_k -\matH'_k -\lambda \matI_k$ & $(m-1)\matI_k$  &$\bfzero_k$& $\cdots$
  &$\bfzero_k$&$\matI_k+ \lambda \matI_k$         \\
   \hline
$\bfzero_k$ & $-(m-1)\matI_k-\lambda \matI_k$  & $(m-2)\matI_k$ &$\cdots$ 
        &  $\bfzero_k$&$\matI_k+\lambda\matI_k$\\
    \hline
$\bfzero_k$ & $\bfzero_k$  & $-(m-2)\matI_k-\lambda \matI_k $ &$\cdots$ & $\bfzero_k$ 
             &$\matI_k+\lambda\matI_k$\\
\hline
$\vdots$ & $\vdots$  & $\vdots$& $\ddots$& $\vdots$& $\vdots$\\
\hline
$\bfzero_k$ & $\bfzero_k$  & $\bfzero_k$& $\ldots$&$2\matI_k$
     &$\matI_k+ \lambda \matI_k$\\
\hline
$\bfzero_k$ & $\bfzero_k$  & $\bfzero_k$& $\ldots$&$-2\matI_k-\lambda \matI_k$&$2\matI_k+\lambda \matI_k$\\
\hline
$\matH_k'$ & $\bfzero_k$  & $\bfzero_k$ &$\ldots$&$\bfzero_k$&$  -\matI_k-\lambda \matI_k$ \\
\hline
\end{tabular}}\, ,

\vskip 1cm
\noindent
which has the same characteristic polynomial as $\matA$. 
Recall the relation $\matH_k - \matH_k' = - m\matI_k$, and so the top left block
of the latter matrix is a diagonal matrix as well. Now add the sum of the first
$m-1$ column-blocks to the $m$th to get the matrix

\vskip 1cm
%{\tiny
{\scriptsize
\begin{tabular}{|c|c|c|c|c|c|c|}
  \hline
$-m\matI_k -\lambda \matI_k$ & $(m-1)\matI_k$  &$\bfzero_k$& $\cdots$
  &$\bfzero_k$&$\bfzero_k$        \\
   \hline
$\bfzero_k'$ & $-(m-1)\matI_k-\lambda \matI_k$  & $(m-2)\matI_k$ &$\cdots$ 
        &  $\bfzero_k$&$\bfzero_k$\\
    \hline
$\bfzero_k$ & $\bfzero_k$  & $-(m-2)\matI_k-\lambda \matI_k $ &$\cdots$ & $\bfzero_k$ 
             &$\bfzero_k$\\
  \hline
$\vdots$ & $\vdots$  & $\vdots$& $\vdots$& $\vdots$& $\vdots$\\
\hline
$\bfzero_k$ & $\bfzero_k$  & $\bfzero_k$& $\ldots$&$-2\matI_k-\lambda \matI_k$&$\bfzero_k$\\
\hline
$\matH'_k$ & $\bfzero_k$  & $\bfzero_k$ &$\ldots$&$\bfzero_k$&$\matH'_k-\matI_k-\lambda \matI_k$\\
\hline
\end{tabular}}\, ,

\vskip 1cm
\noindent
which has the same characteristic polynomial as $\matA$.
Expanding the determinant by the $m$th column of blocks, we get
the characteristic polynomial
$$|\matH_k' - \lambda \matI_k - \matI_k| \, (-m -\lambda)^k (-(m-1) -\lambda)^k\cdots (-2-\lambda)^k =0.$$ 
So, $-r$, for $r=2, 3, \ldots, m$, are eigenvalues, each with
multiplicity~$k$. 

Additional eigenvalues come from the equation
$$ |\matH_k' - \lambda \matI_k - \matI_k| = 0.$$  
These eigenvalues are determined by a method similar to
the one used for the block matrices. In two steps we get
a diagonal matrix, first by multiplying the bottom row by $-1$ and adding 
it to all the blocks above it, then adding the first $k-1$ columns to the last:

\begin{align*} 
&|\matH_k' - \lambda \matI_k - \matI_k| \\
\\
&=
%{\tiny
{\scriptsize
 \begin{vmatrix} 
            -1 -\lambda& 0& 0& \cdots &0& 0&
                       0& \cdots&0& \lambda+1\\
            0& -1 -\lambda& 0& \cdots &0& 0&         
            0& \cdots& 0& \lambda+1\\
             \vdots &\vdots&\vdots&\ddots&\vdots &\vdots&\vdots&\ddots&\vdots&\vdots\\
            0& 0& 0& \cdots &0& 0
                        &0&\cdots& -1 -\lambda& \lambda+1\\
             m n_1& m n_2& m n_3& \ldots &m n_{i-1}& mn_i -m&
            m n_{i+1}& \ldots&mn_{k-1}& m n_k-1-\lambda
\end{vmatrix}}\\
&= {\tiny\begin{vmatrix} 
            -1 -\lambda& 0& 0& \cdots &0& 0&
                       0& \cdots& 0&0\\
            0& -1 -\lambda& 0& \cdots &0& 0&         
            0& \cdots& 0&0\\
             \vdots &\vdots&\vdots&\ddots&\vdots &\vdots&\vdots&\ddots&0 &0\\
             0& 0& 0& \cdots &0& 0&0
                       & \cdots&  -1 -\lambda&0\\
             m n_1& m n_2& m n_3& \ldots &m n_{i-1}& m n_i -m&
            m n_{i+1}& \ldots&mn_{k-1}& \sum_{j=1}^k mn_j-m-1-\lambda
\end{vmatrix}}.
\end{align*}
Expanding the determinant by the last column, we get
the polynomial
\begin{equation}
(-\lambda-1)^{k-1}\Big (m\sum_{j=1}^k  n_j-m- 1 - \lambda\Big) = 0,
\label{Eq:unary}
\end{equation}
so the additional eigenvalues are $-1$ (with multiplicity $k-1$)
and $m\sum_{j=1}^k n_j-m-1 = m\tau_0-m-1$ (with multiplicity 1). 
To summarize, the eigenvalues of $\matA^T$ are the same as
those of $\matA$ which are  shown in Table~\ref{Tab:eigenvalues}.

\begin{table}[h!]
\caption{The eigenvalues and their multiplicities.}
\begin{center}
\begin{tabular}{|c|c|}
  \hline
Eigenvalue& multiplicity\\
  \hline
$\lambda_1= m\tau_0-m-1$& 1\\
$\lambda_2, \ldots, \lambda_k = -1$& $k-1$\\
$\lambda_{k+1}, \ldots,  \lambda_{2k}= -2$& $k$\\
$\lambda_{2k+1}, \ldots, \lambda_{3k} = -3$& $k$\\
$\vdots$&$\vdots$\\
$\lambda_{(m-1)k +1}, \ldots, \lambda_{mk} = -m$& $k$\\
  \hline
\end{tabular}
\end{center}
\label{Tab:eigenvalues}
\end{table}
\subsection{The principal eigenvector}
Let $\vecv_1 = (x_1, \ldots, x_{mk})^T$.
To deal with the matrices at the level of blocks, let us consider $\vecv_1$ as 
$m$ vectors, $\vecy_1, \ldots , \vecy_m$, (each of $k$ components) 
stacked atop of each other. 
That is, $\vecy_{s+1}$ is the segment $(x_ {1+ sk}, \ldots, x_{k+sk})^T$, for $s=0, \ldots, m-1$. 
We now solve

{\footnotesize $$
\begin{tabular}{|c|c|c|c|c|c|c|}
  \hline
$\matH_k^T$ & $(\matH_k')^T$  & $(\matH_k')^T$& $(\matH_k')^T$ & $\cdots$  
             &$(\matH_k')^T$ \\
   \hline
 $(m-1)\matI_k$  & $-(m-1)\matI_k$ &$\bfzero_k$&$\bfzero_k$ &$\cdots$  
             &$\bfzero_k$\\
    \hline
$\bfzero_k$ & $(m-2)\matI_k$  & $-(m-2)\matI_k$&$\bfzero_k$ &$\cdots$ 
             &$\bfzero_k$\\
  \hline
$\bfzero_k$ & $\bfzero_k$  & $(m-3)\matI_k$&$\bfzero_k$&$\cdots$&$\bfzero_k$\\
\hline
$\vdots$ & $\vdots$  & $\vdots$& $\vdots$& $\ddots$& $\vdots$\\
\hline
$\bfzero_k$ & $\bfzero_k$  & $\bfzero_k$&$\bfzero_k$ &$\ldots$&-$\matI_k$\\
\hline
\end{tabular}  \begin{pmatrix}\vecy_1\\
                                 \vecy_2\\
                                  \vdots\\
                                  \vecy_m\end{pmatrix} = \lambda_1
                               \begin{pmatrix}\vecy_1\\
                                 \vecy_2\\
                                  \vdots\\
                                  \vecy_m\end{pmatrix}.
$$}
For $r =2,\ldots, m$, the $r$th row of blocks gives us the equation
$$(m-r+1) \vecy_{r-1} - (m-r+1) \vecy_r = \lambda_1\vecy_r, $$
which is the same as
$$\vecy_r = \frac {m-r+1} {m- r + 1 + \lambda_1} \, \vecy_{r-1}. $$
Thus, recursively, we get all the segments in terms of $\vecy_1$. Namely,
 we have
\begin{equation}
\vecy_r  = \frac {(m-1)_{r-1}} {(m-1 +\lambda_1)_{r-1}} \, \vecy_1
   =  \frac {(m-1)_{r-1}} {(m\tau_0-2)_{r-1}} \, \vecy_1. 
\label{Eq:recalgorithm}
\end{equation}
We can then get $\vecy_1$ from the first row of blocks as the solution
to
$$\matH_k^T\vecy_1 +  (\matH_k')^T\vecy_2 + \cdots + (\matH_k')^T\vecy_m = \lambda_1 \vecy_1.$$
The $i$th row (corresponding to an active raw hook that has not yet
recruited) of this matricial equation reads
\begin{align*}
&(mn_i-m) x_1 + (mn_i-m) x_2 + \cdots + (mn_i-m) x_{i-1} 
         + (mn_i-2m) x_i \\
   & \qquad \qquad \qquad \qquad \qquad \qquad +  (mn_i-m) x_{i+1}  
          + \cdots +  (mn_i-m)x_{mk} = \lambda_1 x_i.       
\end{align*}
Using the fact that $||\vecv_1||_1$
is normalized to 1, we rearrange
to get 
$$m (n_i-1) \sum_{j=1}^{mk}  x_j - mx_i  =  m (n_i-1) - mx_i = \lambda_1 x_i.$$
Plugging in the value of the principal eigenvalue, we obtain
$$x_i = \frac {m(n_i-1)} {m\tau_0 - 1}. $$
We next turn to the $j$th row, for $j\in [k]$, and $j \not = i$ to find
{\begin{align*}
&mn_jx_1 + \cdots +  mn_jx_{j-1}  + 
     m(n_j -1) x_j +  mn_j x_{j+1}  +  \cdots + m n_jx_{mk} = \lambda_1 x_j.
\end{align*}}
By  manipulation similar to the case of $x_i$, we conclude that
$$x_j = \frac {mn_j} {m\tau_0 - 1}. $$

We have determined the segment $\vecy_1$. It is
\begin{equation} 
\vecy_1 = \frac m {m\tau_0 - 1}
\begin{pmatrix} 
            n_1\\
            n_2\\
            \vdots\\
            n_{i-1}\\
            n_i-1\\
            n_{i+1}\\
            \vdots \\
            n_k
\end{pmatrix}.
\label{Eq:topblock}
\end{equation}
\section{Strong laws and joint Gaussian distributions}
\label{Sec:Gaussian} 
Let $D_{n,d_j + sh}$ be the number of 
nodes that are originally of degree $d_j$, and experience latching $s$ times in  an $m$-ary network
at age~$n$, for $j\in[k]$, $s=0, 1, \ldots, m$.
The  urn scheme is associated only with the active nodes. 
An active node of degree $d_j + sh$ has $m-s$ virtual nodes
attached to it. Therefore, we have $D_{n,d_j+sh} = X_{n,j+sk} / (m-s)$, 
with $0 \le s< m$,
where $X_{n, j+sk}$ is the number of balls of color $j+sk$ in the urn after
$n$ draws.  
Inactive nodes cannot recruit any more and do not have corresponding
balls in the urn. 
Some extra work is needed to determine a strong law for inactive nodes.

As we shall see in later examples, the matrix $\matA^T$ may not be
invertible.{\footnote{Invertability excludes only the case $\lambda_1=0$,
which means  $m(\tau_0-1) =1$. This can only happen, if $m=1$ with
$\tau_0 =2$.}} When $(\matA^T)^{-1}$ exists,
we call the network {\em invertible}. 
\begin{theorem}
\label{Thm:strong}
Let $D_{n,d_j + sh}$ be the number of nodes that are originally of degree $d_j$, and experience latching $s$ times in an invertible $m$-ary network
at age~$n$, for $j\in[k]$, $s=0, 1, \ldots, m$, and
let $\vecD_n$ be the vector with these components. Suppose the principal eigenvector
of the urn associated with the network is $\vecv_1 =(x_1, \ldots, x_{mk})^T$. 
The degree vector follows the strong law
 $$ \frac 1 n \vecD_n = \frac 1 n \begin{pmatrix} D_{n,d_1} \\
    D_{n,d_2}\\
    \vdots\\
D_{n,d_k}\\
D_{n,d_1+h}\\
D_{n,d_2+h}\\
\vdots\\
   D_{n,d_k+mh}
\end{pmatrix} \almostsure \vecD^* := \alpha_m 
\begin{pmatrix} 
 \frac {(m-1)_0} {(m\tau_0-2)_0} \, \vecy_1 \\
 \frac {(m-1)_1} {(m\tau_0-2)_1}\, \vecy_1 \\
\vdots \\
\frac {(m-1)_{m-1}} {(m\tau_0-2)_{m-1}}\, \vecy_1\\
\frac  {(m-1)_{m-1}} {\alpha_m (m\tau_0-2)_{m-1}} \, \vecy_1
\end{pmatrix}, $$

where $\alpha_m = m\tau_0-m-1$, and $\vecy_1$ is given in (\ref{Eq:topblock}).
\end{theorem}
\begin{proof}
Let $X_{n,r}$ be the number of virtual nodes of color $r\in [mk]$ in the 
extended network $G_n$ at time $n$ (balls of color $r$ in the urn after $n$ draws), 
and denote by~$\vecX_n$ the vector with these $mk$ components. 

By (\ref{Eq:Athreya}), we have
$$\frac 1 n \, \vecX_n \almostsure   \lambda_1 \vecv_1
        =  (m\tau_0 -m-1) \vecv_1. $$
We determine the principal eigenvector $\vecv_1$ by the recursive algorithm
in~(\ref{Eq:recalgorithm}), 
with~(\ref{Eq:topblock}) at the basis of the
induction. 

The vector of color counts relates to the degrees in the following 
way.
For each active node in the graph of degree $d_j+sh$, for $j\in[k]$ and $s =0, \ldots,m-1$,
there are
$m-s$ virtual nodes of color $j+sk$ attached to it.
So, $D_{n, d_j+sh}$ is $X_{n,j+sk} / (m-s)$, specifying
the first $mk$ components of $\vecD_n$, as in the statement of the theorem.

For the inactive nodes, of degrees $d_1+mh, \dots, d_k+mh$, 
their counts relate
to the principal eigenvector in the following way. 
Let $Y_{n,j}$ be the number of times a ball of 
color $j\in [mk]$ is drawn from the urn
by time $n$, and let~$\vecY_n$ be the vector with these components. A pick from color $r \in [k]$ (with $r\not =i)$ contributes $m n_j$ balls
of color $j$ in the urn, except at $r=j$,  or $r=i$, where the contribution is reduced by
$m$, owing to the hooking. This gives us the relation  
$$X_{n,j} =\sum_{r=1\atop r\not = j, r \not=i}^k m n_j Y_{n,r} + (m n_i-m)
            Y_{n,i} + (m n_j-m)Y_{n,j} + X_{0,j}.$$  
Arguing similarly, when $j=i$, we obtain
$$X_{n,i} =\sum_{r=1\atop  r \not=i}^k m n_j Y_{n,r} + (m n_i - 2m)
            Y_{n,i} + X_{0,i}.$$  
Collected in matrix notation, these relations are
\begin{equation}
\vecX_n = \matA^T \vecY_n + \vecX_0. 
\label{Eq:XnYn}
\end{equation} 
When the inverse of $\matA^T$ exists, we can invert the relation
into 
\begin{equation}
\vecY_n = (\matA^T)^{-1}(\vecX_n - \vecX_0), 
\label{Eq:Yn}
\end{equation} 
and a strong
law for $\vecY_n$ ensues:
$$\frac 1 n\vecY_n = \frac 1 n (\matA^T)^{-1}(\vecX_n - \vecX_0)
           \almostsure (\matA^T)^{-1} (\lambda_1 \vecv_1)
          =  (\matA^T)^{-1} (\matA^T \vecv_1) = \vecv_1 .$$

Note that $Y_{n,j+mk} = D_{n, {d_j+mh}}$. Therefore, 
the last $k$ components of $\vecD_n$ are the last $k$ components
of the principle eigenvector:
$$ \frac 1 n \begin{pmatrix}  
    D_{d_1 + mh}\\
    D_{d_2 + mh}\\
    \vdots\\
   D_{d_k + mh}\\
\end{pmatrix}\almostsure \begin{pmatrix}  x_{(m-1)k+1}\\
    x_{(m-1)k+2}\\
    \vdots\\
   x_{mk}
\end{pmatrix}.$$
\end{proof}
\begin{remark}
\label{Rem:noinverse}
Theorem~\ref{Thm:strong} requires the invertibility of $\matA^T$. 
If $\matA^T$ is not invertible, the associated urn gives us only strong laws
for the degrees of active nodes and determining the number of inactive nodes
requires some other reasoning.  
\end{remark}
\begin{example}
An example of noninvertible networks in Remark~\ref{Rem:noinverse} is a degenerate unary network ($m=1$)
grown out of a path of length 1. This is a degenerate network that grows as a path and the only active
nodes in it are the two vertices at the two ends of the path. In this instance, 
we have $\matA =[0]$, and 
$\matA^T$ has no inverse. The urn argument gives us the strong law $D_{n,1}/n \almostsure 0$. Here, we have the 
relation $D_{n,2} = n Y_{n,1}$, and $D_{n,2}/n \to 1$.
\end{example}

The urn scheme associated with an $m$-ary network is a generalized scheme 
in the sense discussed in Section~\ref{Sec:Polya}. In particular, 
we
have $\Re\, \lambda_2 < \frac 1 2 \Re \, \lambda_1$. So, 
when all the components of the principal eigenvector are positive, 
the central limit
theorem in~(\ref{Eq:Smythe}) applies to the active nodes.
When one of the components of $\vecv_1$ is 0, we call the network 
{\em degenerate}. 
In the present work, a sufficient condition for nondegeneracy is $n_i > 1$.
Note the distinction between degeneracy and invertibility.
\begin{theorem}
\label{Thm:CLT}
Let $D_{n,d_j + sh}$ be the number of nodes that are originally of degree $d_j$, and experience latching $s$ times in a nondegenerate
$m$-ary network at age~$n$, for $j\in[k]$, $s=0, 1, \ldots, m$. 
Let $\vecD_n$ be the vector with these components, and $\vecD^*$
be the almost-sure limit in Theorem~\ref{Thm:strong}. As $n\to\infty$,
we have
 $$\frac 1 {{\sqrt n}}  (\vecD_n - {n}\vecD^*) \ \convD\ \normal(\bfzero_{mk},
                {\bf \Sigma}_{mk}), $$
where ${\bf \Sigma}_{mk}$ is an $mk \times mk$ covariance matrix. 
\end{theorem}
\begin{proof}
We have proved that a central limit theorem applies to the active nodes. By~(\ref{Eq:Yn}), the components of 
$\vecY_n$ are linear combinations of the components 
of $\vecX_n$. It follows that $\vecY_n$ also asymptotically follows 
a multivariate Gaussian law. In particular, 
the  vector of the numbers of inactive nodes (the last $k$ components of $\vecY_n$) follows 
a multivariate Gaussian law and the $mk$ components of $\vecD_n$
are together asymptotically Gaussian. 
\end{proof}
\begin{comment}
\begin{remark}
When $m$ exceeds 2, some of the diagonal elements in $\matA$ are less than $-1$. We cannot apply Janson's algorithm to find the covariance matrix in
the limit normal distribution.  
It is conjectured, but not yet proved, 
that the algorithm does deliver the correct
values in the limiting covariance matrix. The method we describe in Theorem~\ref{Thm:BKM} does not place the requirement that the diagonal elements of the replacement
matrix be at least $-1$.
\end{remark}
\end{comment}
\section{Specific instances with small $m$}
\label{Sec:examples}
The calculations outlined in the previous sections are tractable for small $m$. We consider the unary and binary cases. 
\subsection{Unary hooking networks} 
\label{Subsec:unary}
Consider the case $m=1$. As a corollary to Theorem~\ref{Thm:strong}, we get a strong law for the degrees in a random unary hooking network.
\begin{corollary}
\label{Cor:unarynondegen}
Let $D_{n,d_j + sh}$ be the number of nodes that are originally of degree $d_j$, and experience latching $s$ times  in a nondegenerate unary network
at age~$n$, for $j\in[k]$ and $s=0,1$. Then, we have
$$\frac 1 n \begin{pmatrix} D_{n,d_1} \\
    D_{n,d_2}\\
    \vdots\\
   D_{n,d_k}\\
D_{n,d_1+h}\\
D_{n,d_2+h}\\
    \vdots\\
   D_{n,d_k+h}
\end{pmatrix} \almostsure
\frac 1 {\tau_0-1}
      \begin{pmatrix}
    (\tau_0-2) n_1 \\
    \vdots\\
     (\tau_0-2)  n_{i-1}\\
      (\tau_0-2)  (n_i -1)\\
      (\tau_0-2)  n_{i+1}\\
     \vdots\\
        (\tau_0-2) n_k\\
 n_1 \\
    \vdots\\
     n_{i-1}\\
      n_i -1\\
      n_{i+1}\\
     \vdots\\
        n_k
\end{pmatrix}. $$
\end{corollary}
Note the degenerate case $n_i=1$. In this case, the $i$th component of
the limiting vector is 0. So, $G_0$ has one node of degree $d_i$ (the degree of the hook in the seed),  
and sooner or later it will progress by fusing to become an inactive node of degree~$2d_i$.   
\subsection{A concrete nondegenerate unary case}
Let the seed be the one shown in Figure~\ref{Fig:admit}.
The figure shows the extended seed, with virtual nodes at the
insertion positions colored white (color 1) and blue (color 2).
\begin{figure}[thb]
%\begin{figure}[!htbp]
\begin{center}
\begin{tikzpicture}[scale=0.7]
\node[draw=white] at (-2, -2.55) {hook};
\coordinate (A) at (-2,-2);
\coordinate (B) at (-2,2);
\coordinate (C) at (2,2);
\coordinate (D) at (2,-2);
%%\coordinate (E) at (5,2.5);
%\coordinate (F) at (2.5,5);
\draw [thick] (A)--(B)--(C)--(D)--cycle;
%\draw [thick] (A)--(C);
\draw [thick] (B)--(D);
\draw [thick] (A)--(C);

\draw [thick, fill=black] (-2,-2) circle [radius=0.2];
\draw [thick, fill=black] (-2,2) circle [radius=0.2];
\draw [thick, fill=black] (2,2) circle [radius=0.2];
\draw [thick, fill=black] (2,-2) circle [radius=0.2];
\draw [thick] (3,2) ellipse (1 and 0.3);
\draw [thick] (2,3) ellipse (0.3 and 1);

\coordinate (E) at (-4,2);
\coordinate (F) at (4,-2);
\coordinate (G) at (-4,-2);
\coordinate (H) at (4.3,4.3);
\draw (B)--(E);
\draw (D)--(F);
\draw (A)--(G);
\draw (C)--(H);

\node at (-4,2) [rectangle,draw, fill=white] {$1$};
\node at (4,-2) [rectangle,draw, fill=white] {$1$};
\node at (-4,-2) [rectangle,draw, fill=white] {$1$};
\node at (4.3,4.3) [rectangle,draw, fill=cyan] {$2$};
\end{tikzpicture}
\end{center}
  \caption{An extended seed for a unary network.}
  \label{Fig:admit}
\end{figure}
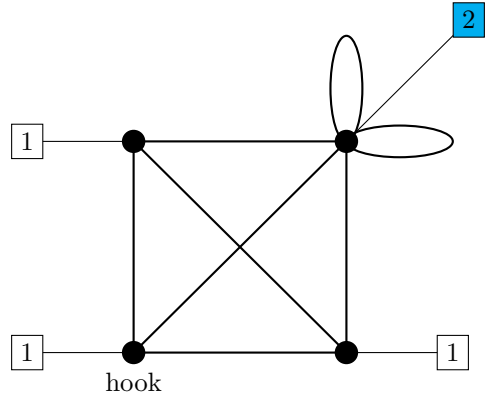

In this example, we have $k=2$, $d_1 = 3$, $d_2 =7$, 
and $\tau_0 = 4$. 
There are $n_1 =3$ nodes of degree 3 and only $n_2 =1$ node of degree 7. The
admissible degrees are $3,6,7,10$. Nodes of degree $d_3 =6$ and 
$d_4 = 10$ are inactive---they 
do not recruit after they appear. The only active nodes are
of degrees 3 and~7. We need only two colors, one (white, color 1) to correspond to
nodes of degree 3 and another (blue, color 2) to correspond to
nodes of degree~7. The replacement matrix is 
$$\matA = \begin{pmatrix} 1 &1\\
                                        2 &0\end{pmatrix}. $$
This is a nondegenerate invertible unary case, as $(\matA^T)^{-1}$ exists. By Corollary~\ref{Cor:unarynondegen}, we have
$$\frac 1 n \begin{pmatrix} 
    D_{n,3}\\
   D_{n,7}\\
D_{n,6}\\
D_{n,10}
\end{pmatrix} \almostsure
\frac 1 3
      \begin{pmatrix}
    4\\
    2\\
    2\\
    1
\end{pmatrix}. $$

In this instance, the replacement matrix is small and amenable to straightforward
variance calculation. In fact, this type of urn is known in the literature as 
a Bagchi--Pal urn, first analyzed in~\cite{Bagchi}.
That paper gives a simple formula for the variance of
the number of white balls.~According to~\cite{Bagchi}, $\Var[D_{n,3}]  =
\Var[X_{n,1}] =1/9$.
We add two balls after each drawing, and we have
$$X_{n,1} + X_{n,2} = D_{n,3} + D_{n,7} = 2n + 4.$$
The linear dependence tells us that $\Var[D_{n,3}] = \Var[D_{n,7}]$ and $\Cov[D_{n,3}, D_{n,7}] = -\Var[X_{n,1}] = -1/9$.  Further, 
by~(\ref{Eq:XnYn}), we have 
\begin{align*}
D_{n,3}  &= Y_{n,1} + 2 Y_{n,2} + 3 = D_{n,6}  + 2 D_{n,10} + 3,\\
D_{n,7}  &= Y_{n,1} + 1 = D_{n,6}   + 1.
\end{align*}
With these relations, we complete the covariance matrix for the central limit theorem:
$$ \frac 1 {{\sqrt n}}
\left (\begin{pmatrix}
         D_{n,3}\\ 
         D_{n,7}\\
 D_{n,6}\\
 D_{n,10} \end{pmatrix} 
     -\frac 1 3 \begin{pmatrix}
         4\\ 
         2\\
         2\\
         1 \end{pmatrix}
\right)    \, \convD\, \normal \left(\begin{pmatrix}
         0\\ 
         0\\
         0\\
         0 \end{pmatrix}, \frac 1 9\begin{pmatrix}
         1&-1&-1&1\\ 
         -1&1&1&-1\\
         -1&1&1&-1\\
         1,&-1&-1&1 \end{pmatrix}\right).$$
\subsection{A case with degeneracy}
\label{Subsec:unarydegen}
Consider the extended seed  shown in Figure~\ref{Fig:degen}, from 
which we build a unary network.
In this example, we have $k=3$ distinct degrees in the seed, 
which are
$d_1 = 1,  d_2 =2$, and $d_3 = 3$. 
There are $n_1 = 1$ node of degree 1, $n_2 =2$ nodes of degree 2, and
$n_3 =1$ node of degree 3. We need three colors for the active nodes, with a ball
of color $j$ corresponding to an active node of degree~$j$, for $j=1,2,3$.

The replacement matrix is 
$$\matA = \begin{pmatrix} -1&2&1\\
                                         0&1&1\\
                                         0&2&0\end{pmatrix}. $$

\bigskip
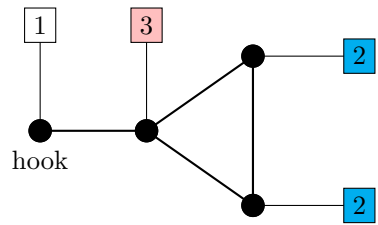
\begin{figure}[thb]
%\begin{figure}[!htbp]
\begin{center}
\begin{tikzpicture}[scale=0.7]
\node[draw=white] at (-4, -0.55) {hook};
\coordinate (A) at (-4,0);
\coordinate (B) at (-2,0);
\coordinate (C) at (0,1.414);
\coordinate (D) at (0,-1.414);
%%\coordinate (E) at (5,2.5);
%\coordinate (F) at (2.5,5);
\draw [thick] (B)--(C)--(D)--cycle;
\draw [thick] (A)--(B);

\draw [thick, fill=black] (-4,0) circle [radius=0.2];
\draw [thick, fill=black] (-2,0) circle [radius=0.2];
\draw [thick, fill=black] (0,1.414) circle [radius=0.2];
\draw [thick, fill=black] (0,-1.414) circle [radius=0.2];

\draw [thick, fill=black] (-4,0) circle [radius=0.2];
\draw [thick, fill=black] (-2,0) circle [radius=0.2];
\draw [thick, fill=black] (0,1.414) circle [radius=0.2];
\draw [thick, fill=black] (0,-1.414) circle [radius=0.2];

\coordinate (E) at (-4,2);
\coordinate (F) at (-2,2);
\coordinate (G) at (2,1.414);
\coordinate (H) at (2,-1.414);
\draw (A)--(E);
\draw (B)--(F);
\draw (C)--(G);
\draw (D)--(H);

\node at (-4,2) [rectangle,draw, fill=white] {$1$};
\node at (-2,2) [rectangle,draw, fill=pink] {$3$};
\node at (2,1.414) [rectangle,draw, fill=cyan] {$2$};
\node at (2,-1.414) [rectangle,draw, fill=cyan] {$2$};
\end{tikzpicture}
\end{center}
  \caption{A seed leading to a degenerate unary network.}
  \label{Fig:degen}
\end{figure}

By~(\ref{Eq:Athreya}), we obtain
$$\frac 1 n \begin{pmatrix} X_{n,1} \\
    X_{n,2}\\
   X_{n,3}
\end{pmatrix}  \almostsure  \frac 1 3\begin{pmatrix} 0 \\
        4 \\
        2
\end{pmatrix}. $$
Let us distinguish nodes of degree $2$ as active, with the count $\hat D_{n,2}$,
and inactive with count $\tilde D_{n,2}$. Likewise, we distinguish 
nodes of degree $3$ as active, with the count $\hat D_{n,3}$,
and inactive with count $\tilde D_{n,3}$. All nodes of degree 4 are
inactive. The actual number of nodes of degree 2 is $D_{n,2} =
  \hat D_{n,2} + \tilde D_{n,2}$, and the actual number of degree 3 is
$D_{n,3} =
  \hat D_{n,3} + \tilde D_{n,3}$.

The network in this example has $\tau_0 =4$. By Corollary~\ref{Cor:unarynondegen}, we have the strong convergence
$$\frac 1 n \begin{pmatrix} D_{n,1} \\
    \hat D_{n,2}\\
   \hat D_{n,3}\\
\tilde D_{n,2}\\
\tilde D_{n,3}\\
 D_{n,4}
\end{pmatrix}  \almostsure  \frac 1 3\begin{pmatrix} 0 \\
        4 \\
        2\\
        0\\
        2\\
        1
\end{pmatrix}, $$
from which we find the limiting vector of degrees: 
$$\frac 1 n \begin{pmatrix} D_{n,1} \\
    D_{n,2}\\
   D_{n,3}\\
 D_{n,4}
\end{pmatrix}  \almostsure  \frac 1 3\begin{pmatrix} 0 \\
        4 \\
        4\\
        1
\end{pmatrix}. $$  

As one of the components of $\vecv_1$ is 0,  
the central limit theorem in Smythe~\cite{Smythe} does
not apply. However, by an alternative formulation in Janson
(Theorem 3.22 in~\cite{Janson}) we still
get a central limit theorem.

For the inactive nodes (of degrees 2, 3 and 4), we use the construction
in~(\ref{Eq:Yn}). Specialized to the unary network at hand, 
the limiting covariance of the vector on the left of~(\ref{Eq:Yn}) follows:
\begin{align*}
\frac 1 n \Cov\left[\begin{pmatrix}
Y_{n,1} \\
Y_{n,2} \\
Y_{n,3} 
\end {pmatrix}\right] &\to  (\matA^T)^{-1} \lim_{n\to\infty} \Big(\frac 1 n\Cov[\vecX_n]\Big) \big( (\matA^T)^{-1}\big)^T\\
&= \frac 1 9 \begin{pmatrix}
              0& 0&0\\
              0& 1& -1\\
              0& -1& 1\\
\end {pmatrix}.
\end{align*}
        
For notational convenience, we refer to the elements of the matrix 
on the right as $y_{r,s}$, for $r,s\in\{1,2,3\}$. 

The three components of the vector on the left-hand side of~(\ref{Eq:Yn}) are 
$Y_{n,1}= \tilde D_{n, 2}$,
$Y_{n,2}=\tilde D_{n,3}$
and $Y_{n,3}=\tilde D_{n,4}$. So, the last matrix is  the bottom right $3\times 3$
block in the full $6 \times 6$ limiting covariance matrix among all the nodes 
(active and inactive). Let us call that latter $6 \times 6$ matrix  $\bf G$,
and refer to its elements by $g_{r,s}$, for $r,s \in [6]$.

As we can only have at most one inactive node of degree 1,
and at most one inactive node of degree 2, we have $\frac 1 n D_{n,1}
\almostsure 0$, and  $\frac 1 n \tilde D_{n,2}
\almostsure 0$. So, 
$$\frac 1 n \Var[D_{n,1}] \to 0, \qquad \frac 1 n \Var[\tilde D_{n,2}] \to 0.$$
Also any limiting covariances involving these degrees are 0.
Reading off~(\ref{Eq:XnYn}) component-wise, we  get three separate equations:
\begin{align}
X_{n,1} &= -Y_{n,1} +1, \label{Eq:Yn1} \\
X_{n,2} &= 2 Y_{n,1} + Y_{n,2} + 2Y_{n,3} + 2, \label{Eq:Yn2} \\
X_{n,3} &= Y_{n,1} +Y_{n,2}  +1 \label{Eq:Yn3}.
\end{align}
From these three equations, we can get the top left $3\times 3$ block of
the full $6\times 6$ limiting covariance.
For instance, by taking the variance of~(\ref{Eq:Yn2}), and scaling
by $n^{-1}$, we get
$$\frac 1 n \Var[X_{n,2}] \to 4 y_{1,1} + y_{2,2} + 4y_{3,3}
                       + 4y_{1,2} + 8y_{1,3} + 4y_{2,3} = \frac 1 9;$$
recall that $y_{1,1} = 0$,
$y_{1,2} = 0$,
 and $y_{1,3} = 0$.

Adding any two of the equations~(\ref{Eq:Yn1})--(\ref{Eq:Yn3}),
gives us one covariance in the top left $3\times 3$ block of
the full $6\times 6$ matrix $\bf G$.

Reorganize~(\ref{Eq:Yn2}) in the form
$$X_{n,2} - Y_{n,2} = 2Y_{n,1} + 2Y_{n,3}  +2,$$
and take the variance, to get
\begin{align*}
&\Var[X_{n,2}] +  \Var[Y_{n,2}] - 2\, \Cov[X_{n,2}, Y_{n,2}] 
   \\ 
& \qquad\qquad = 4\, \Var[Y_{n,1}] + 4\, \Var[Y_{n,3}] 
 + 8 \, \Cov[Y_{n,1}, Y_{n,3}].
\end{align*}
Taking the limit at a scale of $n^{-1}$, we get
$$ \frac 1 n \Cov[\hat D_{n,2}, \tilde D_{n,3}] 
  \to g_{2,5} = - \frac 1 2 (4y_{1,1}+ 4y_{3,3}   + 8 y_{1,3} - y_{2,2} - g_{2,2}) 
    = - \frac 1 9;$$
recall that $y_{1,1} = 0$ and $y_{1,3} = 0$.

Reorganize~(\ref{Eq:Yn2}) in a form separating $2 Y_{n,3}$ on the left,
a similar computation gives $\Cov[\hat D_{n,2},  D_{n,4}]$.
Taking the limit at a scale of $n^{-1}$, we get
$$ \frac 1 n\Cov[\hat D_{n,2}, D_{n,4}]  
  \to g_{2,5}  = -\frac 1 4 (4y_{1,1}+ y_{2,2}   + 4 y_{1,2}  - 4 y_{3,3} - g_{2,2})=  \frac 1 9;$$  
recall that $y_{1,1} = 0$ and $y_{1,2} = 0$.

The trickiest elements of $\bf G$ are $g_{3,5} = g_{5,3}$
and $g_{3,6} = g_{6,3}$. No individual equation 
among~(\ref{Eq:Yn1})--(\ref{Eq:Yn3}) has sufficient information to 
determine these elements, but a combination 
of~(\ref{Eq:Yn2}) and (\ref{Eq:Yn3}) gives us what we want.
Add these two equations, with $2Y_{n,2}$ written on  
the left of~(\ref{Eq:Yn2}),
scale by $n^{-1}$ and take the limits, to get
$$g_{2,2}  + g_{3,3} + 2 g_{2,3}+  4y_{2,2}
    - 4g_{2,6}  - \lim_{n\to\infty} \frac 4 n \Var[\hat D_{n,2}, \tilde D_{n,3}] 
       = 9y_{1,1}+4 y_{3,3}+ 12 y_{1,3};$$ 
recall that $y_{1,1} = 0$ and $y_{1,3} = 0$.
The remaining limit is $g_{3,5} = -\frac 1 9$.
A similar manipulation, with $2Y_{n,3}$ on the left, gives us 
$g_{3,6} = \frac 1 9$. The complete covariance matrix is

$${\bf G} = \lim_{n\to\infty} \frac 1 n \Cov
\left [\begin{pmatrix}
         D_{n,1}\\ 
         \hat D_{n,2}\\
         \hat D_{n,3}\\
   \tilde D_{n,2}\\
         \tilde D_{n,3}\\
 D_{n,4} 
\end{pmatrix} \right]
    = \frac 1 9 
 \begin{pmatrix}
         0&0&0&0&0&0\\ 
         0&1&-1&0&-1&1\\
         0&-1&1&0&-1&1\\
         0&0&0&0&0&0\\
         0&-1&-1&0&1&-1\\
         0&1&1&0&-1&1\\
 \end{pmatrix}.$$

The graph theory point of view does not distinguish
node degrees by history. A network analyst may wish
to make decisions based on the plain degrees. For
instance, for the unary network grown from the seed
in Figure~\ref{Fig:degen},
in a social context all nodes of degree 3 may mean three friends.   
The analyst is only interested in the limiting value of $\frac 1 n\Cov[\vecD_n]$.
The entries of this matrix are related to~$\bf G$. For example
$D_{n,2} =   \hat D_{n,2} + \tilde D_{n,2}$, giving us
the limiting covariance relation
$$ \lim_{n\to\infty} \frac 1 n \Var[D_{n,2}] =   g_{2,2} + g_{4,4} + g_{2,4} = \frac 1 9 .$$
Completing the computations for all the entries of $\lim_{n\to\infty} \frac 1 n\Cov[\vecD_n]$,
we get the central limit theorem
$$ \lim_{n\to\infty} \frac 1 {{\sqrt n}}
\left (\begin{pmatrix}
         D_{n,1}\\ 
         D_{n,2}\\
         D_{n,3}\\
         D_{n,4} 
\end{pmatrix} 
     -\frac 1 3 \begin{pmatrix}
         0\\ 
         4\\
         4\\
         1 \end{pmatrix}
\right)    \, \convD\, \normal \left(\begin{pmatrix}
         0\\ 
         0\\
         0\\
         0 \end{pmatrix}, \frac 1 9 \begin{pmatrix}
         0&0&0&0\\ 
         0&1&-2&1\\
         0&-2&0&0\\
         0&1&0&1 \end{pmatrix}\right).$$
\begin{remark}
Curiously, $\frac 1 n \Var[D_{n,3}] \to 0$. A moment of reflection 
leads us to see that active and inactive nodes of degree 3 appear together
in every step, once the initial hook becomes inactive and the node
of degree 3 in the seed
copy latched into it also recruits. That is, after the latter event, $D_{n,3}
=\hat D_{n,3} + \tilde D_{n,3}$ behaves deterministically.
\end{remark}
\subsection{A binary case}
\label{Subsec:binary}
Consider the extended seed shown in Figure~\ref{Fig:virtual}. 
It is the same seed we considered in the first example on unary
networks but now we build a random binary network out of it. 
In this binary example, we have $k=2$ distinct degrees in the seed, 
which are
$d_1 = 3$, and $d_2 =7$. 
There are $n_1 = 3$ nodes of degree 3 and $n_2 =1$ 
node of degree 7. We need four colors for the active nodes, with 
color $1,2,3,4$ respectively corresponding to the active nodes of degrees $3,7,6,10$.

The replacement matrix is 
$$\matA = \begin{pmatrix}  2&2&1&0\\
                                         4&0&0&1\\
                                         4&2&-1&0\\
                                         4&2&0&-1\end{pmatrix}. $$
According to Theorem~\ref{Thm:strong}, we get
$$ \frac 1 n  \vecD_n^* = \frac 1 n \begin{pmatrix} D_{n,3} \\
   D_{n,7}\\
D_{n,6}\\
D_{n,10}\\
D_{n,9}\\
D_{n,13}
\end{pmatrix}  \almostsure  \frac 1 {21}\begin{pmatrix} 
        60 \\
        30 \\
10\\
5\\
2\\
1
\end{pmatrix} =:\vecD^* . $$

The limiting covariance matrix is computed in the appendix.
The corresponding central limit theorem is
{\small\begin{align*}
& \frac 1 {{\sqrt n}} \left(\begin{pmatrix} D_{n,3} \\
   D_{n,7}\\
D_{n,6}\\
D_{n,10}\\
D_{n,9}\\
D_{n,13}
\end{pmatrix}  - \frac 1 {21}\begin{pmatrix} 
        60 \\
        30 \\
10\\
5\\
2\\
1
\end{pmatrix}\right) %\\
% &\qquad\qquad
\quad  \convD \
\normal  \left(\begin{pmatrix}
         0\\ 
0\\ 
0\\ 
0\\ 
0\\ 
0         
 \end{pmatrix}, \frac 5 {882} \begin{pmatrix}
24& -16&-27&11& 3&5\\ 
-16&20& 11&-19& 5& -1\\ 
-27&11& 40& -8&-13&-3\\ 
11&-19&-8&24& -3&-5\\
3& 5& -13& -3& 10& -2\\
5&-1&-3&-5&-2& 6
 \end{pmatrix}\right).
 \end{align*}}
\section{Concluding remarks}
We studied random $m$-ary networks that grow from a seed, and each
node in the network has $m$ hooking positions. The degrees
in the network evolve over time. Via a connection to the composition
of certain \Polya\ urns, we are able to extract theorems for the degrees
in the $m$-ary network. Namely a strong law (Theorem~\ref{Thm:strong}) and a multivariate central
limit theorem (Theorem~\ref{Thm:CLT}) were developed for $m$-ary networks.
\section*{Appendix}
We illustrate the procedure 
in~(\ref{Eq:Mahmoud21}) on the binary network
in Section~\ref{Subsec:binary}. The limit $\matQ_c$ is  $4 \times 4$, with 
the elements $ q_{ij}$, for $i,j\in [4]$. 

The urn associated with this network has $\lambda_1 = \theta = 5$, and the principal eigenvector is $(x_1, x_2, x_3, x_4)^T = \frac 1 {21} (12,6,2,1)^T$.  
We wish to solve the matricial 
equation
$$ \matA^T \matQ_4+ \matQ_4 \matA + \lambda_1 \matA^T 
         \Diag(x_1, x_2, x_3, x_4)\matA
             - \lambda_1\matA^T \vecv_1 \vecv_1^T \matA 
             - \lambda_1\matQ_4  =
\bfzero_4.$$
We extract elements from the left-hand side and equate them to 0. It
is enough to extract the 10 elements on and above the diagonal, since
$\matQ_4 = [q_{i,j}]_{1\le i,j\le 4}$ 
is symmetric. Going through this extraction, we get
\begin{align*}
-q_{1,1}+8 q_{1,2}+8 q_{1,3}+8 q_{1,4}+\frac {240}{49} &= 0,\\
2 q_{1,1}+2 q_{1,3}+2 q_{1,4}-3q_{1,2}+4q_{2,2}+4q_{2,3} + 4q_{2,4}
            -\frac {160}{49} &=0,\\
 q_{1,1} - 4 q_{1,3} + 4 q_{2,3} +4 q_{3,3}
            + 4 q_{3,4} - \frac {440}{147}&=0,\\
  q_{1,2} - 4 q_{1,4}+4 q_{2,4}+4 q_{3,4}+ 4 q_{4,4}
          + \frac{200}{147}&=0,\\
4q_{1,2} + 4q_{2,3}+4q_{2,4} - 5 q_{2,2} 
           +\frac {200}{49}&=0,\\
q_{1,2} -6 q_{2,3}+2q_{1,3}+ 2 q_{3,3} + 2  q_{3,4}
         + \frac {200}{147}&=0,\\
q_{2,2} - 6 q_{2,4} + 2 q_{1,4} + 2 q_{3,4} + 2 q_{4,4}
               - \frac {320} {147}=0,\\
2q_{1,3} - 7 q_{3,3} + \frac {970}{441} =0,\\
q_{2,3} - 7 q_{3,4} + q_{1,4} -\frac {250} {441} = 0,\\
2q_{2,4} - 7 q_{4,4} + \frac {610} {441} = 0.
\end{align*}
This is a standard system of linear equations, with the 
solution 
$$\lim_{n\to\infty} \frac 1 n \Cov\left [\begin{pmatrix}
 X_{n,1}\\ 
 X_{n,2}\\ 
X_{n,3}\\ 
X_{n,4}        
\end{pmatrix} \right]= \frac 5 {441} \begin{pmatrix}
 48&-32&-27&11\\ 
 -32&40&11&-19\\ 
 -27&11&20&-4\\ 
 11&-19&-4&12\\      
 \end{pmatrix}.$$ 

We remind the reader that colors $1,2,3,4$ correspond to degrees $3,7,6,10$ respectively and that 
$Y_{n,3}$ and $Y_{n,4}$ correspond to $D_{n, 9}$ and $D_{n, 13}$. Additionally, with 
$$D_{n,3}  = \frac 1 2 X_{n,1} , \qquad  D_{n,7} = \frac 1 2 X_{n,2} ,\qquad
D_{n,6} = X_{n,3},  \qquad D_{n,10}= X_{n,4},$$
we have the variances and covariances in the top left $4\times 4$ block:
$$\frac 1 n\Var[D_{n,3}] =\frac 1 {4n} \Var[X_{n,1}] \to \frac{20} {147}, \qquad    
          \frac 1 n \Var[D_{n,7}]= \frac 1 {4n}  \Var[X_{n,2}]  
             \to \frac{50} {441},$$
$$\frac 1 n \Var[D_{n,6}] = \frac 1 n   \Var[X_{n,3}] \to  \frac{100} {441}, \qquad \frac 1 n\Var[D_{n,10}] = \frac 1 n  \Var[X_{n,4}] \to  \frac{20} {147},$$
$$\frac 1 n \Cov[D_{n,3}, D_{n,7}] = \frac 1 {4n} \Cov[X_{n,1}, X_{n,2}] 
         \to  - \frac{40} {441},$$
$$\frac 1 n\Cov[D_{n,3}, D_{n,6}] = \frac 1 {2n} \Cov[X_{n,1}, X_{n,3}] 
       \to  - \frac{15} {98},$$
$$\frac 1 n\Cov[D_{n,3}, D_{n,10}] = \frac 1 {2n} \Cov[X_{n,1}, X_{n,4}] \to                       
               \frac{55} {882},$$
$$\frac 1 n\Cov[D_{n,7}, D_{n,6}] = \frac 1 {2n} \Cov[X_{n,2}, X_{n,3}] \to  \frac{55} {882},$$
$$\frac 1 n\Cov[D_{n,7}, D_{n,10}] = \frac 1 {2n} \Cov[X_{n,2}, X_{n,4}] \to  -\frac{95} {882},$$
$$\frac 1 n \Cov[D_{n,6}, D_{n,10}] = \frac 1 {n} \Cov[X_{n,3}, X_{n,4}] \to  -\frac{20} {441}.$$

By the symmetry of a covariance matrix, we can complete the top left
$4\times 4$ block of the full $6\times 6$ limiting covariance matrix among all the degrees. 

For the inactive nodes (of degrees 9 and 13), we use the construction
in~(\ref{Eq:Yn}). Specialized to the binary network at hand, we
have
$$\vecY_n = (\matA^T)^{-1}\left(
\begin{pmatrix}
               X_{n,1} \\
               X_{n,2} \\
               X_{n,3} \\
               X_{n,4}  
\end {pmatrix}
   - \begin {pmatrix}
                 6 \\
                 2 \\
                 0 \\
                 0
\end {pmatrix}\right).$$
The limiting covariance of the vector on the left, at the scale of $n^{-1}$,  
follows:
\begin{align*}
\frac 1 n \Cov\left[\begin{pmatrix}
Y_{n,1} \\
Y_{n,2} \\
Y_{n,3} \\
Y_{n,4}
\end {pmatrix}\right] &\to  (\matA^T)^{-1}\lim_{n\to\infty}
    \Big(\frac 1 n \Cov[\vecX_n] \Big)\big( (\matA^T)^{-1}\big)^T\\
&= \frac 5 {882}\begin{pmatrix}
              24& -16& -3& -5 \\
             -16& 20& -5& 1  \\
              -3& -5& 10& -2\\
             -5& 1& -2& 6
\end {pmatrix}.
\end{align*}

The bottom two components of the vector on the left-hand side are $Y_{n,3}=D_{n, 9}$ and $Y_{n,4}=D_{n, 13}$. So, the bottom $2\times 2$
block in the last matrix are  the bottom right $2\times 2$
block of the full $6\times 6$ limiting covariance matrix among all the degrees. 

We determine the rest of the full $6\times 6$ matrix from relation~(\ref{Eq:XnYn}),
which in this instance becomes
\begin{equation}
\begin{pmatrix}
X_{n,1}\\
X_{n,2}\\
X_{n,3}\\
X_{n,4}
 \end{pmatrix} = \begin{pmatrix} 
           2Y_{n,1} + 4Y_{n,2}
                  + 4Y_{n,3}+ 4Y_{n,4}+6 \\
                  2Y_{n,1} + 2Y_{n,3}
                  + 2Y_{n,4}+ 2\\
                 Y_{n,1} - Y_{n,3}\\
                  Y_{n,2} - Y_{n,4}
\end{pmatrix}. 
\label{Eq:appendix}
\end{equation}

 Extract the third component and write $Y_{n,3}$
 on the left-hand side.
Taking  the variance, we get
$$\Var[X_{n,3}] +   \Var[Y_{n,3}] +2\, \Cov[D_{n,6} , D_{n,9}]  = \Var[Y_{n,1}].$$
After scaling, we get
\begin{align*} 
 \frac 1 n\Cov[D_{n,6} , D_{n,9} ]  
    &=  \frac 1 {2n}\Var[Y_{n,1}]  -          \frac 1 {2n} \Var[Y_{n,3}]  - 
                       \frac 1 {2n} \Var[X_{n,3}]\\
    &\to \frac 1 2\Big(\frac {20} {147}
           -  \frac{25}  {441}-  \frac{100} {441}\Big)\\
     &= - \frac  {65}{882}.
\end{align*}

The fourth component in~(\ref{Eq:appendix}) can be handled similarly to give the limit
$\lim_{n\to \infty} \frac 1 n\Cov[D_{n,10} , D_{n,13}] = - \frac {25} {882}$. 
Another equation comes from the top component 
in~(\ref{Eq:appendix}). Upon reorganization,
we write
$$X_{n,1} -  4Y_{n,4} = 2 Y_{n,1} + 4Y_{n,2} + 4Y_{n,3} + 6.$$
Taking the variance and using $X_{n,1} = 2 D_{n,3}$,
$Y_{n,4}= D_{n,13}$, we can solve for $\lim_{n\to\infty}
\frac 1 n \Cov[D_{n,3}, D_{n,13}]= \frac {25} {882}$. 

With a different reorganization of the top~components in~(\ref{Eq:appendix}), 
bringing $4Y_{n,3} = 4D_{n,9}$ 
to the left-hand side, we obtain the limit 
$\lim_{n\to\infty} \frac 1 n \Cov[D_{n,3}, D_{n,9}]
= \frac 5 {294}$. 

Coming from the second 
equation from the top in~(\ref{Eq:appendix}), the relation
$$ X_{n,2} = 2 Y_{n,1} + 2 Y_{n,1} +2 Y_{n,1} +2$$
can be handled in a similar manner in two steps, once with $2Y_{n,3}$ on the left, and a second time with  $2Y_{n,4}$ on the left, to respectively
produce the limits 
$$\lim_{n\to\infty} \frac 1 n\Cov[D_{n,7}, D_{n,9}]  
= \frac {25}{882},
 \qquad \lim_{n\to\infty}\frac 1 n \Cov[D_{n,7}, D_{n,13}] = - \frac 5 {882}.$$

No single equation taken from~(\ref{Eq:appendix}) gives us the remaining elements. 
We combine the third and the fourth in the form
$$ X_{n,3} + X_{n,4} = Y_{n,1} + Y_{n,2} -Y_{n,3}- Y_{n,4}.$$
With $Y_{n,3}$ written on the left, we get $\lim_{n\to\infty}\frac 1 n  \Cov[D_{n,9}, D_{n,10}] = - \frac  5 {294}$. The calculation needs the limiting value $\lim_{n\to\infty} \frac 1 n  \Cov[D_{n,6}, D_{n,9}]$,
and this has already been determined.  
A reorganization with $Y_{n,4}$ written on the left, we get $\lim_{n\to\infty}\frac 1 n  \Cov[D_{n,6}, D_{n,13}] = -\frac 5
{294}$. The calculation needs the limiting value $\lim_{n\to\infty}\frac 1 n  \Cov[D_{n,10}, D_{n,13}]$,
and this has already been determined. 
%Ravi: When applicable, I have added DOIs to each of the cited papers.
%\bibliographystyle{plain}
%\begin{thebibliography} {99}
\bibliographystyle{plain}

\begin{thebibliography}{99}
%\end{thebibliography}
%\end{document} 
%%%%%%%%%%%%%%%
\bibitem{Athreya}   
Athreya,  K.\ and Karlin,  S.\ (1968).  
Embedding of urn schemes into continuous time Markov branching 
processes and related limit theorems.  
{\em The Annals of Mathematical Statistics} {\bf 39}, 1801--1817. 
%\href{https://doi.org/10.1214/aoms/1177698013}
\url{https://doi.org/10.1214/aoms/1177698013}.
\bibitem{Bagchi} 
        Bagchi, A.\ and Pal, A.\ (1985). 
        Asymptotic normality in the generalized 
        \Polya-Eggenberger urn model with applications to computer data structures. 
        {\em SIAM Journal on Algebraic and Discrete Methods} {\bf 6}, 394--405.
        %\href{https://doi.org/10.1137/0606041}
        \url{https://doi.org/10.1137/0606041}.
\bibitem{Bahrani}
		Bahrani, M. and Lumbroso, J.\ (2019).
		Split-decomposition trees with prime nodes: 
		Enumeration and random generation of cactus graphs.
		{\em 2018 Proceedings of the Fifteenth 
		Workshop on Analytic Algorithmics and Combinatorics (ANALCO)}, 
		143--157.
		%\href{https://doi.org/10.1137/1.9781611975062.13}
		\url{https://doi.org/10.1137/1.9781611975062.13}.
\bibitem{Bhutani1}	
      Bhutani, K., Kalpathy, R.\ and Mahmoud, H.\ (2021).
		Average measures in polymer graphs.
		{\em International Journal of Computer 
		Mathematics:~Computer Systems Theory} {\bf 6:1}, 
		37--53. %\href{https://doi.org/10.1080/23799927.2020.1860134}
     \url{https://doi.org/10.1080/23799927.2020.1860134}.
\bibitem{Bhutani2}      
          Bhutani, K., Kalpathy, R.\ and  Mahmoud, H. (2022). 
          Random networks grown by fusing edges via urns. 
          {\em Network Science}, 1--14. 
          %\href{https://doi.org/10.1017/nws.2022.30}
          \url{https://doi.org/10.1017/nws.2022.30}.
\bibitem{Bhutani3} 
Bhutani, K., Kalpathy, R.\ and Mahmoud, H.\ (2023).
	Random multi-hooking networks. 
     {\em Probability in the Engineering and Informational 
        Sciences}, 1--15. 
        %\href{https://doi.org/10.1017/s0269964822000523}
	\url{https://doi.org/10.1017/s0269964822000523}.
\bibitem{Bhutani4}
    Bhutani, K., Kalpathy, R., Mahmoud, H. and Ofonedu, A.\ (2023+).
              Some empirical and theoretical attributes of random 
              multi-hooking networks (under review).
\bibitem{Brown}    
		Brown, G. and Shubert, B. (1984).
		On random binary trees.                
		{\em Mathematics of Operations Research} {\bf 9},  43--65. 
		%\href{https://doi.org/10.1287/moor.9.1.43}
		\url{https://doi.org/10.1287/moor.9.1.43}.
\bibitem{ChenChen}
		Chen, C.\ and Mahmoud, H.\ (2016).
		Degrees in random self-similar bipolar networks.
		{\em Journal of Applied Probability} {\bf 53},  434--447.
		%\href{https://doi.org/10.1017/jpr.2016.11}
		\url{https://doi.org/10.1017/jpr.2016.11}.
\bibitem{Holmgren}
		Desmarais, C.\ and Holmgren, C.\ (2019).
		Degree distributions of generalized hooking networks.
		{\em 2019 Proceedings of the Sixteenth Workshop on Analytic Algorithmics and 							Combinatorics (ANALCO)}, 103--110.
		%\href{https://doi.org/10.1137/1.9781611975505.11}
		\url{https://doi.org/10.1137/1.9781611975505.11}.
\bibitem{Colin1} 
		Desmarais, C. and Holmgren, C.\ (2020). 
		Normal limit laws for vertex degrees in randomly grown 
		hooking networks and bipolar networks.
		{\em The Electronic Journal of Combinatorics} {\bf 27(2)}, P2.45.
		%\href{https://doi.org/10.37236/9139}
		\url{https://doi.org/10.37236/9139}.
\bibitem{Colin2} 
		Desmarais, C. and Mahmoud, H.\ (2021). 
		Depths in hooking networks.
	    {\em Probability in the Engineering 
	    and Informational Sciences}, 1--9.
         %\href{https://doi.org/10.1017/s0269964821000164}
         \url{https://doi.org/10.1017/s0269964821000164}.
\bibitem{Gittenberger}
		Drmota, M., Gittenberger, B. and Panholzer, A.\ (2008). 
		The degree distribution of thickened trees. 
		{\em DMTCS Proceedings, Fifth Colloquium on Mathematics and Computer Science} {\bf AI}, 149--162.
		%\href{https://doi.org/10.46298/dmtcs.3561}
		\url{https://doi.org/10.46298/dmtcs.3561}.
\bibitem{Freedman}
Freedman, D.\ (1965).
Bernard Friedman's urn.
    {\em The Annals of Mathematical Statistics} {\bf 36},
    956--970. 
    %\href{https://doi.org/10.1214/aoms/1177700068}
    \url{https://doi.org/10.1214/aoms/1177700068}.
\bibitem{Mohan}
		Gopaladesikan, M., Mahmoud, H.\ and Ward, M.\ (2014).
		Building random trees from blocks.
		{\em Probability in the Engineering and Informational Sciences} {\bf 28}, 67--81.
		%\href{https://doi.org/10.1017/s0269964813000338}
		\url{https://doi.org/10.1017/s0269964813000338}.
\bibitem{Holmgren2}
         Holmgren, C., Janson, S.\ and Sileikis, M.\ (2017).
         Multivariate normal limit laws for the numbers of fringe subtrees in 
         $m$-ary search trees and preferential attachment trees.
         {\em The Electronic Journal of Combinatorics} {\bf 24}, p2.51.
         %\href{https://doi.org/10.37236/6374}
         \url{https://doi.org/10.37236/6374}.
\bibitem{Horn}
      Horn, R.\ and Johnson, C.\ (1985). 
     {\em Matrix Analysis}. Cambridge University Press, Cambridge, UK.
\bibitem{Janson}
Janson, S.\ (2004). 
           Functional limit theorems for multitype branching
           processes and generalized \Polya\  urns. 
           {\em Stochastic Processes and Their Applications}
           {\bf 110}, 177--245.
           %\href{https://doi.org/10.1016/j.spa.2003.12.002}
           \url{https://doi.org/10.1016/j.spa.2003.12.002}.
\bibitem{newJanson}
           Janson, S.\ (2020).
           Mean and variance of balanced \Polya\ urns.
           {\em Advances in Applied Probability}  {\bf 52} , 1224--1248.
           %\href{https://doi.org/10.1017/apr.2020.38}
           \url{https://doi.org/10.1017/apr.2020.38}.
\bibitem{Kalpathy} 
          Kalpathy, R.\ and Mahmoud, H.\ (2016).
          Degree profile of $m$-ary search trees: 
              A vehicle for data structure compression.
         {\em Probability in the Engineering and 
         Informational Sciences} {\bf 30},
         133--123.
         %\href{https://doi.org/10.1017/s0269964815000303}
         \url{https://doi.org/10.1017/s0269964815000303}.
\bibitem{Knuth}
		Knuth, D.\ (1998).
		{\em The Art of Computer Programming, Vol. 3: Sorting and Searching}, 2nd Ed. 					Addison-Wesley, Boston, Massachusetts.
\bibitem{Mybook}
            Mahmoud, H.\ (2008).
            {\em P\' olya Urn Models}. 
            Chapman-Hall, Orlando, Florida.
\bibitem{profile} 
		Mahmoud, H.\ (2019).
		Local and global degree profiles of randomly grown self-similar hooking 
    networks under uniform and preferential attachment.
		{\em Advances in Applied Mathematics} {\bf 111}, 101930.
		%\href{https://doi.org/10.1016/j.aam.2019.07.006}
		\url{https://doi.org/10.1016/j.aam.2019.07.006}.
\bibitem{SP}	
		Mahmoud, H. (2019).
       A spectrum of series-parallel graphs with multiple edge evolution.
       {\em Probability in the Engineering and Informational Sciences} {\bf 33}, 487--499.
       %\href{https://doi.org/10.1017/s0269964818000505}
       \url{https://doi.org/10.1017/s0269964818000505}.
\bibitem{Mahmoud2022}
	      Mahmoud, H. (2021). 
	      Covariances in \Polya\ urn schemes.
          {\em Probability in the Engineering and Informational Sciences} 
          {\bf 37}, 60--71.
          %\href{https://doi.org/10.1017/s0269964821000450}
          \url{https://doi.org/10.1017/s0269964821000450}.
\bibitem{Pouyanne}
           Pouyanne, N.\ (2008).
           An algebraic approach to \Polya\ processes.
         {\em Annales de l'Institut Henri Poincar\'e, Probabilit\'es et Statistiques} 
         {\bf 44}, 293--323.
         %\href{https://doi.org/10.1214/07-aihp130}
         \url{https://doi.org/10.1214/07-aihp130}.
\bibitem{Samorodnitsky}
		Resnick, S. and Samorodnitsky, G.\ (2016). 
		Asymptotic normality of degree counts in a preferential attachment model. 
		{\em Advances in Applied Probability} {\bf 48}, 283--299.  
		%\href{https://doi.org/10.1017/apr.2016.56}
		\url{https://doi.org/10.1017/apr.2016.56}.
\bibitem {Smythe} 
          Smythe, R.\ (1996). Central limit theorems for urn models.
          {\em Stochastic Processes and Their Applications}
          {\bf 65}, 115--137.
          %\href{https://doi.org/10.1016/s0304-4149(96)00094-4}
          \url{https://doi.org/10.1016/s0304-4149(96)00094-4}.
\bibitem{VanderHofstad}
		van der Hofstad, R.\ (2016). 
		{\em Random Graphs and Complex Networks, Vol. 1}. 
		Cambridge University Press, Cambridge, UK. 
%%%%%%%%%%%Good until here%%%%%%%%%%%%
\end{thebibliography}

\end{document}